\documentclass[12pt]{amsart}
\usepackage[utf8x]{inputenc}
\usepackage[T1]{fontenc}
\usepackage{lmodern}
\title{Internally {C}alabi--{Y}au algebras and cluster-tilting objects}
\author{Matthew Pressland}
\address{Matthew Pressland\\Max-Planck-Institut für Mathematik\\Vivatsgasse 7
 \\53111~Bonn\\Germany}
\email{mdpressland@bath.edu}
\date{December 28, 2016} 
\usepackage[hmargin=3cm,vmargin=3.5cm]{geometry}
\usepackage{amsmath}
\usepackage{amssymb}
\usepackage{amsfonts}
\usepackage{amsthm}
\usepackage{mathrsfs}
\usepackage{booktabs}
\usepackage[usenames,dvipsnames,svgnames]{xcolor}
\usepackage[colorlinks, urlcolor=Navy, linkcolor=Navy, citecolor=Navy]{hyperref}
\usepackage[abbrev,alphabetic,msc-links]{amsrefs}
\usepackage{color}
\usepackage{todonotes}
\usepackage{tikz}
\usepackage{tikz-cd}
\usetikzlibrary{arrows}
\pgfarrowsdeclarecombine{twohead}{twohead}
{angle 90}{angle 90}{angle 90}{angle 90}

\usepackage[capitalise,compress]{cleveref}
\Crefname{eg}{Example}{Examples}
\Crefname{thm}{Theorem}{Theorems}
\Crefname{conj}{Conjecture}{Conjectures}
\Crefname{prop}{Proposition}{Propositions}
\Crefname{lem}{Lemma}{Lemmas}
\Crefname{defn}{Definition}{Definitions}
\Crefname{equation}{Equation}{Equations}
\Crefname{thm*}{Theorem}{Theorems}
\Crefname{conj*}{Conjecture}{Conjectures}
\Crefname{ques}{Question}{Questions}
\Crefname{rem}{Remark}{Remarks}

\theoremstyle{plain}
\newtheorem{thm}{Theorem}[section]
\newtheorem{thm*}{Theorem}
\newtheorem{lem}[thm]{Lemma}
\newtheorem{cor}[thm]{Corollary}
\newtheorem{prop}[thm]{Proposition}

\theoremstyle{definition}
\newtheorem{defn}[thm]{Definition}
\newtheorem{eg}[thm]{Example}

\newtheorem{rem}[thm]{Remark}
\theoremstyle{remark}

\newcommand{\CC}{\mathbb{C}}
\newcommand{\KK}{\mathbb{K}}

\newcommand{\ZZ}{\mathbb{Z}}

\newcommand{\bP}{\mathbf{P}}

\newcommand{\typeA}[1]{\mathsf{A}_{#1}}

\newcommand{\afftypeA}[1]{\widetilde{\mathsf{A}}_{#1}}

\newcommand{\cat}{\mathcal{C}}

\newcommand{\frobcat}{\mathcal{E}}
\newcommand{\abcat}{\mathcal{A}}
\newcommand{\ctiltcat}{\mathcal{T}}

\newcommand{\mut}{\mathrm{m}}

\DeclareMathOperator{\Sub}{Sub}
\DeclareMathOperator{\stabSub}{\underline{Sub}}

\DeclareMathOperator{\CM}{CM}
\DeclareMathOperator{\stabCM}{\underline{CM}}
\DeclareMathOperator{\Mod}{Mod}
\DeclareMathOperator{\fgmod}{mod}

\DeclareMathOperator{\Ext}{Ext}
\DeclareMathOperator{\Hom}{Hom}
\DeclareMathOperator{\stabHom}{\underline{Hom}}
\DeclareMathOperator{\stabEnd}{\underline{End}}
\DeclareMathOperator{\RHom}{\mathbf{R}Hom}
\DeclareMathOperator{\SL}{SL}

\DeclareMathOperator{\End}{End}
\DeclareMathOperator{\catproj}{proj}
\DeclareMathOperator{\catinj}{inj}
\DeclareMathOperator{\add}{add}

\DeclareMathOperator{\GP}{GP}

\DeclareMathOperator{\coker}{coker}
\DeclareMathOperator{\gldim}{gl.dim}
\DeclareMathOperator{\pdim}{p.dim}

\DeclareMathOperator{\im}{im}

\DeclareMathOperator{\stabGP}{\underline{GP}}

\DeclareMathOperator{\per}{per}

\DeclareMathOperator{\thick}{thick}

\newcommand{\op}{\mathrm{op}}
\newcommand{\otherwise}{\text{otherwise}}

\newcommand{\iso}{\cong}
\newcommand{\isoto}{\stackrel{\sim}{\to}}

\newcommand{\node}{\bullet}
\newcommand{\tens}{\mathbin{\otimes}}

\newcommand{\ltens}{\mathbin{\stackrel{\mathbf{L}}{\otimes}}}

\newcommand{\dsum}{\mathbin{\oplus}}
\newcommand{\bigdsum}{\bigoplus}

\newcommand{\union}{\cup}

\newcommand{\id}[1]{1_{#1}}

\renewcommand{\emptyset}{\varnothing}

\newcommand{\cohom}[2]{\mathrm{H}^{#1}(#2)}

\newcommand{\powser}[2]{#1[\hspace{-0.1em}[#2]\hspace{-0.1em}]}
\newcommand{\Grass}[2]{G_{#1}^{#2}}

\newcommand{\Endalg}[2]{\End_{#1}(#2)^{\op}}
\newcommand{\stabEndalg}[2]{\stabEnd_{#1}(#2)^{\op}}

\newcommand{\comp}[1]{\widehat{#1}}
\newcommand{\vout}[1]{#1^+}
\newcommand{\vin}[1]{#1^-}
\newcommand{\head}[1]{h#1}
\newcommand{\tail}[1]{t#1}
\newcommand{\idemp}[1]{e_{#1}}

\newcommand{\preproj}[1]{\Pi(#1)}

\newcommand{\env}[1]{{#1}^{\varepsilon}}
\newcommand{\dcat}[1]{\mathcal{D}{#1}}
\newcommand{\dcatfd}[1]{\mathcal{D}_{\mathrm{fd}}(#1)}
\newcommand{\bdcat}[1]{\mathcal{D}^b{#1}}
\newcommand{\resdcat}[2]{\mathcal{D}_{#1}(#2)}
\newcommand{\resdcatfd}[2]{\mathcal{D}_{\mathrm{fd},#1}(#2)}

\newcommand{\inj}[1]{Q_{#1}}

\newcommand{\simp}[1]{S_{#1}}

\newcommand{\res}[1]{\mathbf{P}(#1)}
\newcommand{\rhom}[1]{\Omega_{#1}}

\newcommand{\Ctilt}[2]{\widetilde{A}}

\newcommand{\jac}[2]{\mathcal{J}(#1,#2)}
\newcommand{\frjac}[3]{\mathcal{J}(#1,#2,#3)}

\newcommand{\cc}[1]{\varphi_{#1}}

\newcommand{\lift}[1]{\widetilde{#1}}
\newcommand{\rad}[1]{\mathfrak{m}(#1)}

\newcommand{\Kdual}{\mathrm{D}}
\newcommand{\dual}[1]{#1^\vee}
\newcommand{\contra}[1]{#1^*}
\newcommand{\syz}[1]{\Omega_{#1}}

\newcommand{\stab}[1]{\underline{#1}}

\newcommand{\set}[1]{\{#1\}}
\newcommand{\Span}[1]{\langle#1\rangle}

\newcommand{\der}[1]{\partial_{#1}}

\newcommand{\Div}[1]{\Delta_{#1}}
\newcommand{\rel}[1]{\rho_{#1}}

\newcommand{\close}[1]{\overline{#1}}

\newcommand{\noop}[1]{}

\newcommand{\clustalg}[1]{\mathscr{A}(#1)}

\newcommand{\frozen}{dashed}

\begin{document}

\begin{abstract}
We describe what it means for an algebra to be internally $d$-Calabi--Yau with respect to an idempotent. This definition abstracts properties of endomorphism algebras of $(d-1)$-cluster-tilting objects in certain stably $(d-1)$-Calabi--Yau Frobenius categories, as observed by Keller--Reiten. We show that an internally $d$-Calabi--Yau algebra satisfying mild additional assumptions can be realised as the endomorphism algebra of a $(d-1)$-cluster-tilting object in a Frobenius category. Moreover, if the algebra satisfies a stronger `bimodule' internally $d$-Calabi--Yau condition, this Frobenius category is stably $(d-1)$-Calabi--Yau. We pay special attention to frozen Jacobian algebras; in particular, we define a candidate bimodule resolution for such an algebra, and show that if this complex is indeed a resolution, then the frozen Jacobian algebra is bimodule internally $3$-Calabi--Yau with respect to its frozen idempotent. These results suggest a new method for constructing Frobenius categories modelling cluster algebras with frozen variables, by first constructing a suitable candidate for the endomorphism algebra of a cluster-tilting object in such a category, analogous to Amiot's construction in the coefficient-free case.
\end{abstract}
\maketitle

\section{Introduction}

Cluster categories, first introduced in special cases by Buan--Marsh--Reineke--Reiten--Todorov \cite{buantilting} and later generalised by Amiot \cite{amiotcluster}, are certain $\Hom$-finite $2$-Calabi--Yau triangulated categories that model the combinatorics of cluster algebras without frozen variables. In particular, a cluster category $\cat$ contains cluster-tilting objects, which are objects $T$ satisfying
\[\add{T}=\set{X\in\cat:\Ext^1_\cat(T,X)=0}=\set{X\in\cat:\Ext^1_\cat(X,T)=0}.\]
Basic cluster-tilting objects, whose summands in any direct sum decomposition are pairwise non-isomorphic, model the clusters of the cluster algebra; from now on, any time we refer to a cluster-tilting object, we assume it to be basic.

The main reason that the cluster-tilting objects of $\cat$ can be said to model the clusters of a cluster algebra is that, as with a cluster, it is possible to pass from one cluster-tilting object to another by a process of mutation. For any indecomposable summand $T_i$ of a cluster-tilting object $T$, there exists a unique indecomposable $T_i'\in\cat$, not isomorphic to $T_i$, such that $T/T_i\dsum T_i'$ is a cluster-tilting object. Moreover, $T_i'$ can be computed by either of the exchange triangles
\[\begin{tikzcd}[column sep=20pt,row sep=3pt]
T_i\arrow{r}{f}&X_i\arrow{r}&T_i'\arrow{r}&T_i[1],\\
T_i'\arrow{r}&Y_i\arrow{r}{g}&T_i\arrow{r}&T_i'[1],
\end{tikzcd}\]
in which $f$ is a minimal left $\add{T/T_i}$-approximation of $T_i$, and $g$ is a minimal right $\add{T/T_i}$-approximation of $T_i$. Choosing an initial cluster-tilting object $T^0=\bigdsum_{i=1}^nT_i^0$ of $\cat$ yields a  cluster character \cite{calderocluster}, \cite{palucluster} $\cc{}\colon\cat\to\clustalg{Q,\mathbf{x}}$ from the objects of $\cat$ to the cluster algebra $\clustalg{Q,\mathbf{x}}$ with initial seed given by the quiver $Q$ of $\Endalg{\cat}{T}$ with cluster variables $x_i=\cc{T_i^0}$. The exchange triangles correspond to the exchange relations
\[\cc{T_i}\cc{T_i'}=\cc{X_i}+\cc{Y_i}\]
in this cluster algebra. The cluster variables of $\clustalg{Q,\mathbf{x}}$ are precisely the elements of the form $\cc{M}$ for $M$ an indecomposable reachable rigid object of $\cat$, where rigid means that $\Ext^1_\cat(M,M)=0$, and reachable means that $M$ is a summand of a reachable cluster-tilting object, i.e.\ one obtained from $T^0$ by a finite sequence of mutations. The clusters are the sets of the form $\set{\cc{T_1},\dotsc,\cc{T_n}}$ for $T=\bigdsum_{i=1}^nT_i$ reachable cluster-tilting in $\cat$.

For a more thorough introduction to the theory of cluster algebras and their categorification, we recommend Keller's survey \cite{kellercluster}.

The categorification of cluster algebras by cluster categories has proved to be very useful in studying their combinatorics, since the cluster category $\cat$ can be considered more globally than the associated cluster algebra $\clustalg{Q,\mathbf{x}}$. For example, to identify clusters or cluster variables of $\clustalg{Q,\mathbf{x}}$, one usually has to find a sequence of mutations from a known cluster, which is a highly computationally intensive procedure. By contrast, cluster-tilting or rigid objects of $\cat$ are characterised intrinsically.

For this reason, it would be extremely useful to be able to more readily categorify cluster algebras that do have frozen variables, particularly as most of the examples occurring in nature, such as those on the coordinate rings of partial flag varieties and their unipotent cells, as studied by Geiß--Leclerc--Schröer \cite{geisspartial}, are of this type. The natural candidate for such a categorification is a stably $2$-Calabi--Yau Frobenius category, as we now describe.

A Frobenius category is an exact category with enough projective and injective objects, such that these two classes of objects coincide. If $\frobcat$ is a Frobenius category, then the stable category $\stab{\frobcat}=\frobcat/\catproj{\frobcat}$ is triangulated by a famous result of Happel \cite{happeltriangulated}*{\S I.2}. It is immediate from the definition that if $T\in\frobcat$ is cluster-tilting, then $\catproj{\frobcat}\subset\add{T}$. In this case, we must have $\catproj{\frobcat}=\add{P}$ for some object $P=\bigdsum_{i=r+1}^nT_i$; the intention is that the objects $T_i$ for $r<i\leq n$, which occur as summands of every cluster-tilting object of $\frobcat$, will correspond to the frozen variables of a cluster algebra. Factoring out $\catproj{\frobcat}$ corresponds to setting these frozen variables to $1$ in the cluster algebra, to recover a cluster algebra without frozen variables. Thus the stable category $\stab{\frobcat}$ should be a cluster category---in particular it should be $2$-Calabi--Yau.

Such categorifications have been described for certain cluster algebras relating to partial flag varieties, for example by Geiss--Leclerc--Schröer \cite{geisspartial}, Jensen--King--Su \cite{jensencategorification} and Demonet--Iyama \cite{demonetlifting}, and we will recall some of these constructions in Section~\ref{frobclustcats}. Nájera Chávez \cite{najerachavez2calabiyau} has also categorified finite type cluster algebras with `universal' coefficients. However, in all cases, the construction of the category depends on having at least a partial understanding of the overall structure of the cluster algebra. Thus the methods of these papers cannot be easily abstracted to produce categorifications for more general cluster algebras with frozen variables.

The main aim of this paper is to consider how one might be able to produce a categorification of a cluster algebra with frozen variables without understanding this global structure, instead starting only from the data of a single seed, which is how a cluster algebra is usually specified. This is analogous to Amiot's construction of cluster categories in the case that there are no frozen variables; given a seed, one has to find a rigid potential on the quiver of the seed such that the resulting Jacobian algebra (see \Cref{bimodcomplex}) is finite dimensional, and then Amiot provides a general recipe for constructing a categorification of the cluster algebra. Our construction in the case that there are frozen variables is similar, but requires more data satisfying more conditions. Given the seed of a cluster algebra with frozen variables, we take its quiver, and aim to add arrows between frozen variables and choose a potential such that the resulting frozen Jacobian algebra satisfies a number of conditions, most importantly that of being bimodule internally $3$-Calabi--Yau. If this can be achieved, the general machinery developed in this paper can take over to produce the desired Frobenius category, from which the frozen Jacobian algebra can be recovered as the endomorphism algebra of a cluster-tilting object.

For most of the paper, we will in fact work in a higher level of generality, and construct stably $d$-Calabi--Yau Frobenius categories admitting $d$-cluster-tilting objects (\Cref{ctiltdef}); setting $d=2$ recovers the definition of cluster-tilting given above.

The structure of the paper is as follows. In \Cref{defns} we say what it means for an algebra to be internally $d$-Calabi--Yau with respect to an idempotent $e$ (\Cref{idcy}), and also make a stronger, more symmetric, definition of bimodule internally $d$-Calabi--Yau with respect to $e$ (\Cref{bimodidcy}). An algebra $A$ satisfying either of these definitions has, in particular, finite global dimension, and a duality
\[\Kdual\Ext_A^i(M,N)=\Ext_A^{d-i}(N,M)\]
for any pair $M$ and $N$ of $A$-modules such that $M$ is finite dimensional, $eM=0$ and both $M$ and $N$ are perfect when considered as stalk complexes in the bounded derived category of $A$. In much of the paper we will also assume that $A$ is Noetherian, in which case this last condition reduces to $M$ and $N$ being finitely generated.

In \Cref{frobclustcats}, we introduce the class of Frobenius $m$-cluster categories for $m\geq1$, and exploit a result of Keller--Reiten \cite{kellerclustertilted} to show that the endomorphism algebra of an $m$-cluster-tilting object in a Frobenius $m$-cluster category is internally $(m+1)$-Calabi--Yau. A Morita-type theorem of Iyama--Kalck--Wemyss--Yang \cite{iyamafrobenius} implies that any Frobenius $m$-cluster category admitting an $m$-cluster-tilting object with Noetherian endomorphism algebra is equivalent to the category of Gorenstein projective modules over some Iwanaga--Gorenstein ring (\Cref{gpdef}), and so we pay special attention to categories of this form. We also give brief descriptions of some important families of Frobenius $m$-cluster categories already appearing in the literature.

In \Cref{idcytodct}, we show that an algebra $A$ that is internally $d$-Calabi--Yau with respect to an idempotent $e$ (on both sides) for some $d\geq2$, and satisfies mild additional assumptions, necessarily arises as the endomorphism algebra of a $(d-1)$-cluster-tilting object in some Frobenius category determined by $A$ and $e$. Precisely, we prove the following theorem.

\begin{thm*}[\Cref{airanalogue}]
\label{airanalogueintro}
Let $A$ be a Noetherian algebra, and let $e\in A$ be an idempotent such that $A/\Span{e}$ is finite dimensional and both $A$ and $A^{\op}$ are internally $d$-Calabi--Yau with respect to $e$. Write $B=eAe$ and $\stab{A}=A/\Span{e}$. Then
\begin{itemize}
\item[(i)]$B$ is Iwanaga--Gorenstein with Gorenstein dimension at most $d$, so
\[\GP(B)=\set{X\in\fgmod{B}:\Ext^i_B(X,B)=0,\ i>0}\]
is a Frobenius category,
\item[(ii)]$eA$ is $(d-1)$-cluster-tilting in $\GP(B)$, and
\item[(iii)]there are natural isomorphisms $\Endalg{B}{eA}\isoto A$ and $\Endalg{\stabGP(B)}{eA}\isoto\stab{A}$.
\end{itemize}
\end{thm*}

Under the stronger assumption that $A$ is bimodule internally $d$-Calabi--Yau with respect to $e$, we can show more.

\begin{thm*}[\Cref{bimodicytostabcy}]
Let $A$ be a Noetherian algebra and let $e\in A$ be an idempotent such that $A/\Span{e}$ is finite dimensional, and $A$ is bimodule internally $d$-Calabi--Yau with respect to $e$. Write $B=eAe$. Then all of the conclusions of \Cref{airanalogueintro} hold, and moreover $\stabGP(B)$ is $(d-1)$-Calabi--Yau.
\end{thm*}

While, in general, checking that an algebra is internally $d$-Calabi--Yau with respect to an idempotent can be very difficult, there is more hope in the case that $A$ is a frozen Jacobian algebra (\Cref{frjacalg}). Such an algebra is presented via a quiver with relations, in which the relations are dual to some of the arrows; the arrows which do not have any corresponding relations are called frozen, and their end-points are frozen vertices. In \Cref{bimodcomplex}, we show that the required Calabi--Yau symmetry can be deduced from the exactness of a combinatorially defined complex $\res{A}\to A$ (\Cref{frjacres}), generalising work of Ginzburg \cite{ginzburgcalabiyau} for Jacobian algebras. More precisely, we prove the following.

\begin{thm*}[\Cref{frjaci3cy}]
\label{frjaci3cyintro}
If $A$ is a frozen Jacobian algebra such that $\res{A}$ is quasi-isomorphic to $A$, then $A$ is bimodule internally $3$-Calabi--Yau with respect to the idempotent $e=\sum_{v\in F_0}\idemp{v}$, where $F_0$ denotes the set of frozen vertices.
\end{thm*}

The results of Section~\ref{bimodcomplex} are inspired by work of Broomhead \cite{broomheaddimer} on consistent dimer models (also known as brane tilings or bipartite field theories) on closed surfaces, which has applications to theoretical physics. We expect our results to have consequences for the more recent theory of dimer models on surfaces with boundary, studied for example by Franco \cite{francobipartite}, and which has already appeared in the context of cluster categorification for the Grassmannian in work of Baur--King--Marsh \cite{baurdimer}. Dimer models on surfaces with boundary have also appeared, under the name `plabic graphs', in work of Postnikov \cite{postnikovtotal} on the positive Grassmannian, and in recent work of Goncharov \cite{goncharovideal}.

\section*{Notation and conventions}

Throughout, we fix a field $\KK$, and assume all categories are $\KK$-linear and all algebras are associative $\KK$-algebras with unit. If $V$ is a $\KK$-vector space, we write $\Kdual{V}=\Hom_\KK(V,\KK)$ for the dual space. All modules are unital, and are left modules unless otherwise indicated. Given an algebra $A$, we denote by $\Mod{A}$ the category of all $A$-modules, and by $\fgmod{A}$ the category of finitely generated $A$-modules. Given an object $X$ of an exact category $\frobcat$, we denote by $\add{X}$ the full subcategory of $\frobcat$ with objects isomorphic to finite direct sums of direct summands of $X$.  We say an algebra $A$ is Noetherian if it is Noetherian as both a left and right module over itself; this is stronger than requiring it to be Noetherian as an $A$-bimodule. We denote by $\dcat{A}$ and $\bdcat{A}$ the derived and bounded derived categories of $A$, and by $\dcatfd{A}$ the full subcategory of $\bdcat{A}$ consisting of objects with finite dimensional total cohomology. We denote by $\per{A}$ the thick triangulated subcategory of $\bdcat{A}$ generated by $A$ (or equivalently by the finitely generated projective $A$-modules). Objects of $\per{A}$ are called perfect, and we will say that an $A$-module $M$ is perfect if the stalk complex with $M$ in degree $0$ is perfect.

\section*{Acknowledgements}
This work is based on part of my Ph.D.\ thesis \cite{presslandfrobenius}*{\S5}, which I worked on at the University of Bath under the supervision of Alastair King, and with the support of an Engineering and Physical Sciences Research Council studentship. I am very grateful to Alastair for his support and encouragement. I also thank Charlie Beil, Martin Kalck, Joe Karmazyn and Dong Yang for useful conversations. This paper was revised during a stay at Universität Bielefeld funded by Sonderforschungsbereich 701; I am grateful to Henning Krause and the rest of the representation theory group in Bielefeld for their hospitality. After helpful comments from the referee, it was further revised during a stay at the Max-Planck-Institut für Mathematik in Bonn, funded by the Max-Planck-Gesellschaft.

\section{\texorpdfstring{Internally $d$-Calabi--Yau algebras}{Internally d-Calabi--Yau algebras}}
\label{defns}

This section introduces our main definitions.

\begin{defn}
\label{idcy}
Let $A$ be a $\KK$-algebra, $e$ an idempotent of $A$, and $d$ a non-negative integer. We say $A$ is \emph{internally $d$-Calabi--Yau} with respect to $e$ if
\begin{itemize}
\item[(i)]$\gldim{A}\leq d$, and
\item[(ii)]for each $i\in\ZZ$, there is a functorial duality
\[\Kdual\Ext_A^i(M,N)=\Ext_A^{d-i}(N,M)\]
where $M$ and $N$ are perfect $A$-modules such that $M$ is also a finite dimensional $A/\Span{e}$-module, where $\Span{e}=AeA$ is the two-sided ideal of $A$ generated by $e$.
\end{itemize}
\end{defn}

Recall that a module $M$ is perfect if the stalk complex with $M$ in degree $0$ is perfect, i.e.\ quasi-isomorphic to a bounded complex of finitely generated projective $A$-modules. Thus a perfect module is always both finitely generated and finitely presented, but the converse statements need not hold. If $A$ is Noetherian with finite global dimension (which will be the case for much of the paper, including the main theorems in \Cref{idcytodct}), then an $A$-module is perfect if and only it is finitely generated. We will also consider other situations in which perfectness is equivalent to a more familiar notion, and will try to indicate these as they arise.

Note the lack of symmetry in the finiteness conditions on $M$ and $N$; cf.\ \cite{kellercalabiyau}*{Lem.~4.1}. These conditions are imposed in order to force the space $\Ext_A^{d-i}(N,M)$ on the right-hand side of the duality formula to be finite dimensional over $\KK$.

\begin{rem}
If $A$ is internally $d$-Calabi--Yau with respect to $e$ then it is internally $d$-Calabi--Yau with respect to $e+e'$ for any idempotent $e'\in A$ orthogonal to $e$. An algebra $A$ is $d$-Calabi--Yau if and only if it is internally $d$-Calabi--Yau with respect to $0$, and hence with respect to every idempotent. At the other extreme, $A$ is internally $d$-Calabi--Yau with respect to $1$ if and only if $\gldim{A}\leq d$.
\end{rem}

\begin{rem}
\label{fdidcy}
A finite dimensional algebra $A$ is internally $d$-Calabi--Yau with respect to $e$ if and only if the same is true of $A^{\op}$; since $A$ is finite dimensional, it is Noetherian, so $\gldim{A}=\gldim{A^{\op}}$ \cite{weibelintroduction}*{Ex.~4.1.1}, and $\Kdual=\Hom_\KK(-,\KK)$ induces an equivalence $\fgmod{A^{\op}}\isoto(\fgmod{A})^{\op}$ yielding the required functorial duality for $A^{\op}$.
\end{rem}

\Cref{idcy} is not necessarily left-right symmetric in this way if $A$ is infinite dimensional, so we will also make a stronger definition that does have this property. Denote by $\env{A}=A\tens_\KK A^{\op}$ the enveloping algebra of $A$, so that an $A$-bimodule is the same as an $\env{A}$-module. Write $\rhom{A}=\RHom_{\env{A}}(A,\env{A})$. We view $\rhom{A}$ as a complex in $\dcat{\env{A}}$ via the `inner' multiplication on $\env{A}$; for any homomorphism $f\colon M\to\env{A}$ of $A$-bimodules such that $f(m)=u\tens v$ and any $x\tens y\in\env{A}$, let $xfy(m)=uy\tens xv$.

Recall \cite{amiotstable}*{Defn.~2.1} that $A$ is said to be \emph{bimodule $d$-Calabi--Yau} if $A\in\per{\env{A}}$ (i.e.\ $A$ is quasi-isomorphic to a bounded complex of finitely generated projective $A$-bimodules) and there is an isomorphism $A\isoto\rhom{A}[d]$ in $\dcat{\env{A}}$. This definition is slightly weaker than that of Ginzburg \cite{ginzburgcalabiyau}*{3.2.5}, as we will not need to impose any `self-duality' condition on the isomorphism.

If $A$ is an algebra with quotient $\stab{A}$, write $\resdcat{\stab{A}}{A}$ for the full subcategory of $\dcat{A}$ consisting of complexes with homology groups in $\Mod{\stab{A}}$, and $\resdcatfd{\stab{A}}{A}$ for the full subcategory of $\resdcat{\stab{A}}{A}$ consisting of objects with finite dimensional total cohomology.
\begin{defn}
\label{bimodidcy}
An algebra $A$ is \emph{bimodule internally $d$-Calabi--Yau} with respect to an idempotent $e\in A$ if
\begin{itemize}
\item[(i)]$\pdim_{\env{A}}{A}\leq d$,
\item[(ii)]$A\in\per{\env{A}}$, and
\item[(iii)]there exists a triangle
\[\begin{tikzcd}[column sep=20pt]
A\arrow{r}{\psi}&\rhom{A}[d]\arrow{r}&C\arrow{r}&A[1]
\end{tikzcd}\]
in $\dcat{\env{A}}$, such that
\[\RHom_A(C,M)=0=\RHom_{A^{\op}}(C,N)\]
for any $M\in\resdcatfd{\stab{A}}{A}$ and $N\in\resdcatfd{\stab{A}^{\op}}{A^{\op}}$, where $\stab{A}=A/\Span{e}$.
\end{itemize}
\end{defn}

\begin{rem}
An algebra $A$ is bimodule internally $d$-Calabi--Yau with respect to $0$ if and only if $\psi$ can be chosen to be a quasi-isomorphism, or equivalently if $A$ is bimodule $d$-Calabi--Yau. In this case, (i) follows from (ii) and (iii) \cite{amiotstable}*{Prop.~2.5(b)}. When $e\ne0$, this implication does not hold, and so we must make the stronger condition part of the definition. Note also that (i) does not imply (ii), since it asserts only that $A$ has a finite resolution by projective $\env{A}$-modules, whereas (ii) requires additionally that these modules are finitely generated.

An algebra $A$ is bimodule internally $d$-Calabi--Yau with respect to $1$ if and only if $\pdim_{\env{A}}{A}\leq d$ and $A\in\per{\env{A}}$; in this case $\stab{A}=0$, so condition (iii) is satisfied for any $\psi$.
\end{rem}

\begin{rem}
\label{lrsym}
There is an isomorphism $\env{A}\isoto\env{(A^{\op})}$ given by reversing the order of the tensor product. The resulting equivalence $\fgmod{\env{A}}\isoto\fgmod{\env{(A^{\op})}}$ takes $A$ to $A^{\op}$ (and $\stab{A}$ to $\stab{A}^{\op}$). As a result, Definition~\ref{bimodidcy} is left-right symmetric, meaning that $A$ is bimodule internally $d$-Calabi--Yau with respect to $e$ if and only if the same is true of $A^{\op}$.
\end{rem}

The following lemma, due to Keller, allows us to recover dualities of extension groups between $A$-modules from bimodule properties of $A$.

\begin{lem}[\cite{kellercalabiyau}*{Lem.~4.1}]
\label{ltensrhom}
Assume $A\in\per{\env{A}}$. For all objects $M,N\in\dcat{A}$ such that $M$ has finite dimensional total cohomology, there is a functorial isomorphism
\[\Kdual\Hom_{\dcat{A}}(M,N)\isoto\Hom_{\dcat{A}}(\rhom{A}\ltens_AN,M).\]
\end{lem}

If $A$ is bimodule $d$-Calabi--Yau, then $\rhom{A}\iso A[-d]$ in $\dcat{\env{A}}$. It then follows from \Cref{ltensrhom} that for any $M,N\in\Mod{A}$, with $M$ finite dimensional, we have
\begin{align*}
\Kdual\Ext^i_A(M,N)&=\Kdual\Hom_{\dcat{A}}(M,N[i])\\
&\iso\Hom_{\dcat{A}}(\rhom{A}\ltens_AN[i],M)\\
&\iso\Hom_{\dcat{A}}(N[i-d],M)\\
&=\Ext^{d-i}_A(N,M).
\end{align*}
We now use \Cref{ltensrhom} to prove a similar result for bimodule internally $d$-Calabi--Yau algebras.

\begin{thm}
\label{bimodcydual}
Let $A$ be bimodule internally $d$-Calabi--Yau with respect to $e$, and let $\stab{A}=A/\Span{e}$. Then for any $N\in\dcat{A}$ and any $M\in\resdcatfd{\stab{A}}{A}$, we have a functorial isomorphism
\[\Kdual\Hom_{\dcat{A}}(M,N)=\Hom_{\dcat{A}}(N[-d],M).\]
\end{thm}

\begin{proof}
Pick a triangle
\[\begin{tikzcd}[column sep=20pt]
A[-d]\arrow{r}{\psi}&\rhom{A}\arrow{r}&C\arrow{r}&A[1-d]
\end{tikzcd}\]
by applying $[-d]$ to a triangle as in \Cref{bimodidcy}. Applying $-\ltens_AN$ yields a triangle
\[\begin{tikzcd}[column sep=20pt]
N[-d]\arrow{r}&\rhom{A}\ltens_AN\arrow{r}&C\ltens_AN\arrow{r}&N[1-d]
\end{tikzcd}\]
in $\dcat{A}$. Now apply $\RHom_A(-,M)$, to get a triangle
\[\begin{tikzcd}[column sep=-20pt]
\RHom_A(\rhom{A}\ltens_AN,M)\arrow{rr}&&\RHom_A(N[-d],M)\arrow{dl}{[1]}\\
&\RHom_A(C\ltens_AN,M)\arrow{ul}
\end{tikzcd}\]
Since $M\in\resdcatfd{\stab{A}}{A}$, we have $\RHom_A(C,M)=0$ by definition, and so
\[\RHom_A(C\ltens_AN,M)=\RHom_A(N,\RHom_A(C,M))=0.\]
Thus $\RHom_A(\rhom{A}\ltens_AN,M)\iso\RHom_A(N[-d],M)$. We obtain the desired result by taking $0$-th cohomology and applying \Cref{ltensrhom}.
\end{proof}

\begin{cor}
\label{bimodtocat}
If $A$ is bimodule internally $d$-Calabi--Yau with respect to $e$, then it is internally $d$-Calabi--Yau with respect to $e$, and moreover we have a functorial isomorphism
\[\Kdual\Ext_A^i(M,N)=\Ext_A^{d-i}(N,M)\]
for any finite dimensional $M\in\fgmod{A/\Span{e}}$ and any $N\in\Mod{A}$, which need not be perfect as $A$-modules.
\end{cor}
\begin{proof}
Since $\pdim_{\env{A}}{A}\leq d$, there is an exact sequence
\[\begin{tikzcd}[column sep=20pt]
0\arrow{r}&P_d\arrow{r}&\dotsb\arrow{r}&P_1\arrow{r}&P_0\arrow{r}&A\arrow{r}&0
\end{tikzcd}\]
of $A$-bimodules, in which each $P_i$ is a projective $A$-bimodule. If $X$ is any $A$-module, then $P_i\tens_AX$ is a projective $A$-module, and so applying $-\tens_AX$ to the above sequence gives a projective resolution
\[\begin{tikzcd}[column sep=20pt]
0\arrow{r}&P_d\tens_AX\arrow{r}&\dotsb\arrow{r}&P_1\tens_AX\arrow{r}&P_0\tens_AX\arrow{r}&X\arrow{r}&0
\end{tikzcd}\]
of $X$. It follows that $\gldim{A}\leq d$.

Now by \Cref{bimodcydual}, if $N\in\Mod{A}$ and $M\in\fgmod{A/\Span{e}}$ is finite dimensional, we have
\begin{align*}
\Kdual\Ext^i_A(M,N)&=\Kdual\Hom_{\dcat{A}}(M,N[i])\\
&=\Hom_{\dcat{A}}(N[i-d],M)\\
&=\Ext^{d-i}_A(N,M).
\end{align*}
In particular, this duality applies when $M$ and $N$ are perfect $A$-modules, so $A$ is internally $d$-Calabi--Yau with respect to $e$.
\end{proof}

While \Cref{bimodcydual} and \Cref{bimodtocat} show that bimodule internally Calabi--Yau algebras have stronger properties than those required by the definition of internally Calabi--Yau, with the duality formula applying to a wider class of objects, we do not currently know of any examples of internally Calabi--Yau algebras which are not bimodule internally Calabi--Yau. It seems unlikely  that the two classes of algebras coincide, but it is tempting to speculate based on work of Bocklandt \cite{bocklandtgraded} that if $A$ is an internally Calabi--Yau algebra admitting a suitable grading, then it is also bimodule internally Calabi--Yau.

\begin{eg}
\label{i3cyeg}
Consider the algebra $A=\KK Q/I$ given by the quiver
\[Q=\mathord{\begin{tikzpicture}[baseline=0]
\node at (0,0.5) (1) {$1$};
\node at (2,0.5) (2) {$2$};
\node at (1,-0.5) (3) {$3$};
\path[-angle 90,font=\scriptsize]
	(1) edge node[above] {$\alpha_1$} (2)
	(2) edge node[below right] {$\alpha_2$} (3)
	(3) edge node[below left] {$\alpha_3$} (1);
\end{tikzpicture}}\]
with ideal of relations $I=\Span{\alpha_2\alpha_1,\alpha_1\alpha_3}$. This is an example of a frozen Jacobian algebra; see \Cref{frjacalg} and \Cref{frjacalgeg}. One can check (for example, by \Cref{frjaci3cy} below) that $A$ is bimodule internally $3$-Calabi--Yau with respect to $\idemp{1}+\idemp{2}$, and so both $A$ and $A^{\op}$ are internally $3$-Calabi--Yau with respect to this idempotent.

Similarly, the algebra $A'=\KK Q'/I'$ given by the quiver
\[Q'=\mathord{\begin{tikzpicture}[baseline=0]
\node at (-2,-0.75) (1) {$1$};
\node at (0,0.75) (2) {$2$};
\node at (2,-0.75) (3) {$3$};
\node at (0,-0.75) (4) {$4$};
\path[-angle 90,font=\scriptsize]
	(1) edge node[above left] {$\alpha_1$} (2)
	(2) edge node[right] {$\alpha_3$} (4)
	(4) edge node[below] {$\alpha_4$} (1)
	(4) edge node[below] {$\alpha_5$} (3)
	(3) edge node[above right] {$\alpha_2$} (2);
\end{tikzpicture}}\]
with ideal of relations $I'=\Span{\alpha_3\alpha_1,\alpha_1\alpha_4-\alpha_2\alpha_5,\alpha_3\alpha_2}$ is bimodule internally $3$-Calabi--Yau with respect to $\idemp{1}+\idemp{2}+\idemp{3}$.
\end{eg}

\section{\texorpdfstring{Frobenius $m$-cluster categories}{Frobenius m-cluster categories}}
\label{frobclustcats}

This section is devoted to describing a class of categories, which we term Frobenius $m$-cluster categories, providing us with a rich source of internally $(m+1)$-Calabi--Yau algebras. Indeed, certain categories in this class motivated \Cref{idcy}.

Recall that an exact category $\frobcat$ is Frobenius if it has enough projective objects and enough injective objects, and $\catproj{\frobcat}=\catinj{\frobcat}$. By a famous result of Happel \cite{happeltriangulated}*{\S I.2}, the stable category $\stab{\frobcat}=\frobcat/\catproj{\frobcat}$ is triangulated, with suspension given by the inverse syzygy $\syz{}^{-1}$, taking an object to the cokernel of an injective hull.

\begin{defn}
Let $\frobcat$ be a Frobenius category, and let $m\geq1$. We say that $\frobcat$ is \emph{stably $m$-Calabi--Yau} if its stable category $\stab{\frobcat}$ is $m$-Calabi--Yau, meaning that $\stab{\frobcat}$ is $\Hom$-finite, and there is a functorial duality
\[\Kdual\stabHom_\frobcat(X,Y)=\stabHom_\frobcat(Y,\syz{}^{-m}X)\]
for all $X,Y\in\frobcat$. Here $\stabHom_\frobcat(X,Y)$ denotes the space of morphisms $X$ to $Y$ in $\stab{\frobcat}$.
\end{defn}

\begin{defn}
\label{ctiltdef}
Let $\frobcat$ be an exact category, and let $\ctiltcat\subset\frobcat$ be a full and functorially finite subcategory closed under direct sums and direct summands. We say $\ctiltcat$ is an \emph{$m$-cluster-tilting subcategory} if
\[\set{X\in\frobcat:\Ext^i_\frobcat(X,\ctiltcat)=0,\:0<i<m}=\ctiltcat=\set{X\in\frobcat:\Ext^i_\frobcat(\ctiltcat,X)=0,\:0<i<m}.\]
Here `$\Ext^i_\frobcat(X,\ctiltcat)=0$' is taken to mean `$\Ext^i_\frobcat(X,T)=0$ for all $T\in\ctiltcat$'. We say an object $T\in\frobcat$ is an \emph{$m$-cluster-tilting object} if $\add{T}$ is an $m$-cluster-tilting subcategory.
\end{defn}

If $\frobcat$ is a Frobenius category, then for any $X,Y\in\frobcat$ and $i>0$, we have
\[\Ext^i_\frobcat(X,Y)=\stabHom_\frobcat(X,\syz{}^{-i}Y).\]
Thus if $\frobcat$ is stably $m$-Calabi--Yau, the two equalities appearing in \Cref{ctiltdef} are equivalent to one another.

\begin{defn}
\label{frobclustcat}
Let $\frobcat$ be a Frobenius category, and $m\geq1$. Then $\frobcat$ is called a \emph{Frobenius $m$-cluster category} if it is idempotent complete, stably $m$-Calabi--Yau, and $\gldim{\Endalg{\frobcat}{T}}\leq m+1$ for any $m$-cluster-tilting object $T$, of which there is at least one. If these properties hold for $m=2$, then $\frobcat$ will be called simply a \emph{Frobenius cluster category}.
\end{defn}

Note that while $\stab{\frobcat}$ is $\Hom$-finite for any Frobenius $m$-cluster category $\frobcat$, we do not assume that $\frobcat$ itself is $\Hom$-finite (cf.\ \Cref{jkscats}). By the following result, which follows immediately from a theorem of Keller--Reiten \cite{kellerclustertilted}*{5.4}, Frobenius $m$-cluster categories provide a rich source of examples of internally Calabi--Yau algebras.

\begin{thm}
\label{ctiltidcy}
Let $\frobcat$ be a Frobenius $m$-cluster category, let $T\in\frobcat$ be a basic $m$-cluster-tilting object, and write $A=\Endalg{\frobcat}{T}$. Let $P$ be a maximal projective-injective summand of $T$, and let $e\in A$ be the idempotent given by projection onto $P$. Then $A$ is internally $(m+1)$-Calabi--Yau with respect to $e$.
\end{thm}
\begin{proof}
Let $M$ and $N$ be perfect $A$-modules, and assume that $M$ is a finite dimensional $A/\Span{e}$-module. Then by \cite{kellerclustertilted}*{5.4}, there is a functorial duality
\begin{align*}
\Ext^i_A(M,N)&=\Hom_{\per{A}}(M[-i],N)\\
&=\Kdual\Hom_{\per{A}}(N,M[m+1-i])=\Kdual\Ext^{m+1-i}_A(N,M)
\end{align*}
for all $i$. Since $\gldim{A}\leq m+1$ by assumption, we have the desired result.
\end{proof}

\begin{rem}
If $\frobcat$ is a Frobenius $m$-cluster category such that $A=\Endalg{\frobcat}{T}$ is Noetherian for any $m$-cluster tilting object $T$, then $\frobcat^{\op}$ is a Frobenius $m$-cluster category with the same $m$-cluster-tilting objects as $\frobcat$; the extra assumption is used to ensure that
\[\gldim A^{\op}=\gldim{A}\leq m+1.\]
Thus, under these circumstances, $A^{\op}$ is also internally $(m+1)$-Calabi--Yau with respect to $e$.
\end{rem}

In the context of \Cref{ctiltidcy}, we have a more straightforward characterisation of the $A/\Span{e}$-modules that are perfect as $A$-modules, again coming from Keller--Reiten's work.

\begin{prop}
Let $\frobcat$ be a Frobenius $m$-cluster category, let $T\in\frobcat$ be a basic $m$-cluster-tilting object, and write $A=\Endalg{\frobcat}{T}$. Let $P$ be a maximal projective-injective summand of $T$, let $e\in A$ be the idempotent given by projection onto $P$ and write $\stab{A}=A/\Span{e}$. Then an $\stab{A}$-module $M$ is perfect as an $A$-module if and only if it is finitely generated.
\end{prop}
\begin{proof}
Let $M\in\fgmod{\stab{A}}$. If $M$ is perfect over $A$, then it is in particular finitely generated over $A$, and thus (by the same generating set) over $\stab{A}$. Since $\stab{A}=\stabEndalg{\frobcat}{T}$ and $\stab{\frobcat}$ is $\Hom$-finite, $\stab{A}$ is finite dimensional. It follows that any finitely generated $\stab{A}$-module is finitely presented over $\stab{A}$, and thus perfect over $A$ by \cite{kellerclustertilted}*{5.4(c)}.
\end{proof}

The following proposition, which is based on work of Iyama \cite{iyamaauslander}*{\S2}, \cite{iyamahigher}*{Thm.~3.6.2}, gives sufficient conditions on a Frobenius category $\frobcat$ for us to conclude that $\gldim{\Endalg{\frobcat}{T}}\leq m+1$ for any $m$-cluster-tilting object $T\in\frobcat$, as well as providing a more straightforward characterisation of the perfect modules over this endomorphism algebra.

\begin{prop}
\label{gldimbound}
Let $\frobcat\subseteq\abcat$ be a full, extension closed, Frobenius subcategory of an abelian category $\abcat$ such that
\begin{itemize}
\item[(i)]$\abcat$ has enough projectives,
\item[(ii)]$\frobcat$ contains $\catproj{\abcat}$,
\item[(iii)]$\frobcat$ is closed under kernels of epimorphisms, and
\item[(iv)]for any $X\in\abcat$, and any exact sequence
\[\begin{tikzcd}[column sep=20pt]
0\arrow{r}&Y\arrow{r}&P_m\arrow{r}&\cdots\arrow{r}&P_0\arrow{r}&X\arrow{r}&0
\end{tikzcd}\]
with $P_i\in\catproj{\abcat}$, we have $Y\in\frobcat$.
\end{itemize}
Let $T\in\frobcat$ be $m$-cluster-tilting, and write $A=\Endalg{\frobcat}{T}$. Then an $A$-module $M$ is perfect if and only if it is finitely presented, and for such $M$ we have $\pdim_A{M}\leq m+1$. It follows that if $A$ is Noetherian, then $\gldim{A}\leq m+1$.
\end{prop}
\begin{proof}
As already noted, any perfect $A$-module is finitely presented. For the converse, let $M$ be a finitely presented $A$-module with presentation
\[\begin{tikzcd}[column sep=20pt]
\Hom_\frobcat(T,T_1)\arrow{r}{f_*}\arrow{r}&\Hom_\frobcat(T,T_0)\arrow{r}&M\arrow{r}&0,
\end{tikzcd}\]
for some $T_0,T_1\in\add{T}$, and form the exact sequence
\[\begin{tikzcd}[column sep=20pt]
0\arrow{r}&K_1\arrow{r}&T_1\arrow{r}{f}&T_0
\end{tikzcd}\]
in $\abcat$. By \cite{iyamaauslander}*{Prop.~2.6}, the subcategory $\frobcat$ is contravariantly finite in $\abcat$, and hence so is $\add{T}$. Working inductively for $1\leq j\leq m-1$, let $r_j\colon T_{j+1}\to K_j$ be a right $\add{T}$-approximation of $K_j$, and let $K_{j+1}$ be its kernel (in $\abcat$). By additionally defining $K_0=\im{f}$ and $K_{-1}=\coker{f}$, we obtain exact sequences
\[\begin{tikzcd}[column sep=20pt]
0\arrow{r}&K_j\arrow{r}&T_j\arrow{r}&K_{j-1}\arrow{r}&0
\end{tikzcd}\]
for all $0\leq j\leq m$. Combining these to form the exact sequence
\[\begin{tikzcd}[column sep=20pt]
0\arrow{r}&K_{m}\arrow{r}&T_{m}\arrow{r}&\cdots\arrow{r}&T_0\arrow{r}&K_{-1}\arrow{r}&0,
\end{tikzcd}\]
we may use \cite{iyamaauslander}*{Prop.~2.6} again to see that $K_{m}\in\frobcat$.

Next we check that $K_{m}\in\add{T}$, which we do by checking that $\Ext^i_\frobcat(T,K_{m})=0$ for $0<i<m$. For any $j$, we may apply $\Hom_\frobcat(T,-)$ to the short exact sequence
\[\begin{tikzcd}[column sep=20pt]
0\arrow{r}&K_{j+1}\arrow{r}&T_{j+1}\arrow{r}&K_j\arrow{r}&0
\end{tikzcd}\]
to find, using that $\Ext^i_\frobcat(T,T_i)=0$ for $0<i<m$, that $\Ext^i_\frobcat(T,K_{j+1})=\Ext^{i-1}_\frobcat(T,K_j)$ for $1<i<m$. Moreover, if $j\geq1$, the map $T_{j+1}\to K_j$ in the above sequence is the right $\add{T}$-approximation $r_j$, for which $\Hom_\frobcat(T,r_j)$ is surjective, so $\Ext^1_\frobcat(T,K_{j+1})=0$. It follows that for $0<i<m$, we have
\[\Ext^i_\frobcat(T,K_{m})=\Ext^{i-1}_\frobcat(T,K_{m-1})=\dotsb=\Ext^1_\frobcat(T,K_{m-i+1})=0,\]
as required.

It follows from the above calculations that the sequence
\[\begin{tikzcd}[column sep=20pt]
0\arrow{r}&\Hom_\frobcat(T,K_{j+1})\arrow{r}&\Hom_\frobcat(T,T_{j+1})\arrow{r}&\Hom_\frobcat(T,K_j)\arrow{r}&0
\end{tikzcd}\]
is exact for all $j\geq1$. Thus, writing $T_{m+1}=K_m\in\add{T}$, applying $\Hom_\frobcat(T,-)$ to the sequence
\[\begin{tikzcd}[column sep=20pt]
0\arrow{r}&T_{m+1}\arrow{r}&T_{m}\arrow{r}&\cdots\arrow{r}&T_1\arrow{r}{f}&T_0
\end{tikzcd}\]
yields a projective resolution of $M$ by finitely generated projective $A$-modules, and so $M$ is perfect with $\pdim_A{M}\leq m+1$.

For the final statement, if $A$ is Noetherian, then every finitely generated $A$-module is also finitely presented, and thus has projective dimension at most $m+1$, as required.
\end{proof}

We now recall some work of Iyama--Kalck--Wemyss--Yang \cite{iyamafrobenius}, which gives a normal form for a Frobenius $m$-cluster category that is `small enough', in the sense it admits an $m$-cluster-tilting object with Noetherian endomorphism algebra. It will follow from this description that any such Frobenius $m$-cluster category may be embedded into an abelian category in such a way that all of the assumptions of \Cref{gldimbound} hold. We begin with the following definitions.

\begin{defn}
\label{gpdef}
An algebra $B$ is \emph{Iwanaga--Gorenstein} if it is Noetherian and has finite injective dimension as both a left and right module over itself. In this case, the left and right injective dimensions coincide, and are called the \emph{Gorenstein dimension} of $B$. For brevity, an Iwanaga--Gorenstein algebra with Gorenstein dimension $d$ will be called \emph{$d$-Iwanaga--Gorenstein}. If $B$ is such an algebra, we write
\[\GP(B)=\set{X\in\fgmod{B}:\Ext^i_B(X,B)=0,\ i>0}=\syz{}^d(\fgmod{B})\]
for the category of \emph{Gorenstein projective} $B$-modules. This is a full, extension closed subcategory of $\fgmod{B}$, and is a Frobenius category under the inherited exact structure \cite{buchweitzmaximal}*{\S4.8}. This category is sometimes \cite{amiotstable}, \cite{jensencategorification} denoted by $\CM(B)$, and the objects called maximal Cohen--Macaulay $B$-modules, but we do not do this since for certain choices of $B$ there exist definitions of Cohen--Macaulay not coinciding with this one \cite{iyamafrobenius}*{Rem.~3.3}.
\end{defn}

If $B$ is an $(m+1)$-Iwanaga--Gorenstein algebra, then the assumptions of \Cref{gldimbound} hold for the exact subcategory $\frobcat=\GP(B)$ of the abelian category $\abcat=\fgmod{B}$. Thus the endomorphism algebra of an $m$-cluster-tilting object of $\GP(B)$ has global dimension at most $m+1$ whenever it is Noetherian.

\begin{thm}[\cite{iyamafrobenius}*{Thm.~2.7}]
\label{kiwythm}
Let $\frobcat$ be an idempotent complete Frobenius category such that $\catproj{\frobcat}=\add{P}$ for some $P\in\frobcat$. Assume there exists $M\in\frobcat$ such that $P\in\add{M}$, the endomorphism algebra $A=\Endalg{\frobcat}{M}$ is Noetherian, and $\gldim{A}=d<\infty$. Then $B=\Endalg{\frobcat}{P}$ is Iwanaga--Gorenstein of Gorenstein dimension at most $d$, and there is an equivalence
\[\Hom_\frobcat(P,-)\colon\frobcat\isoto\GP(B).\]
\end{thm}

\begin{cor}
\label{normalform}
Let $\frobcat$ be a Frobenius $m$-cluster category and let $T\in\frobcat$ be an $m$-cluster-tilting object such that $\Endalg{\frobcat}{T}$ is Noetherian. Let $P$ be a maximal projective summand of $T$, and write $B=\Endalg{\frobcat}{P}$. Then $B$ is Iwanaga--Gorenstein of Gorenstein dimension at most $m+1$, and $\frobcat\simeq\GP(B)$.
\end{cor}
\begin{proof}
By definition, $\frobcat$ is idempotent complete and $\gldim{\Endalg{\frobcat}{T}}\leq m+1$. Since $T$ is $m$-cluster-tilting, $\catproj{\frobcat}\subseteq\add{T}$, and so we have $\catproj{\frobcat}=\add{P}$. Now the result follows from \Cref{kiwythm}.
\end{proof}

Under the notation and assumptions of \Cref{normalform}, let $e$ be the idempotent of $A=\Endalg{\frobcat}{T}$ given by projection onto $P$, with respect to which $A$ is internally $(m+1)$-Calabi--Yau by \Cref{ctiltidcy}. Then the algebra $B=\Endalg{\frobcat}{P}$ is the idempotent subalgebra $eAe$. We will see below that the Gorenstein dimension of $B$ may be strictly less than $m+1$. An interesting question suggested by \Cref{normalform}, to which we have no good answer at this stage, is the following: can one find reasonable conditions on an Iwanaga--Gorenstein algebra $B$ under which $\GP(B)$ is a Frobenius $m$-cluster category?

The remainder of the section is devoted to examples. We describe two families of examples of Frobenius cluster categories, one $\Hom$-finite and the other $\Hom$-infinite, to which \Cref{ctiltidcy} applies to show that the endomorphism algebras of cluster-tilting objects are internally $3$-Calabi--Yau with respect to projection onto a maximal projective-injective summand. We also give a family of examples of Frobenius $1$-cluster categories arising as part of the algebraic McKay correspondence.

\begin{eg}
\label{birscats}
Buan--Iyama--Reiten--Scott \cite{buancluster} construct a family of $\Hom$-finite stably $2$-Calabi--Yau Frobenius categories $\Sub{\Pi_\omega}$. Here $\Pi=\preproj{\Delta}$ is the preprojective algebra associated to a graph $\Delta$, and $\omega$ is a finite product of simple reflections in the Weyl group of $\Delta$. The algebra $\Pi_\omega$ is a (finite dimensional) quotient of $\Pi$, and $\Sub{\Pi_\omega}$ is the full subcategory of $\fgmod{\Pi_\omega}$ given by objects isomorphic to submodules of direct sums of copies of $\Pi_\omega$. Then $\Sub{\Pi_\omega}$ is closed under extensions and subobjects (in particular under kernels of epimorphisms), and contains $\catproj{\Pi_\omega}$ and $\syz{}(\fgmod{\Pi_\omega})$. Since $\Sub{\Pi_\omega}$ is $\Hom$-finite, it follows from \Cref{gldimbound} that it is a Frobenius cluster category.

We note that the categories $\Sub{\Pi_\omega}$ constructed by Buan--Iyama--Reiten--Scott contain an important class of categories considered by Geiß--Leclerc--Schröer \cite{geisspartial}, which we will also describe. For $\Delta$ a Dynkin diagram, let $\Pi$ be the preprojective algebra of type $\Delta$. If $j$ is a node of $\Delta$, write $\inj{j}$ for the injective $\Pi$-module with socle at $j$. Then for any subset $J$ of the nodes of $\Delta$, we may write $\inj{J}=\bigdsum_{j\in J}\inj{j}$, and consider the category $\Sub{\inj{J}}$ of $\Pi$-modules isomorphic to a submodule of a direct sum of copies of $Q_J$, or equivalently of those $\Pi$-modules with socle supported on $J$. The category $\Sub{\inj{J}}$ models a cluster algebra structure on the coordinate ring of a dense open subset of the partial flag variety attached to the data of $\Delta$ and $J$. For example, when $\Delta$ is of type $\typeA{n}$ and $J$ consists of a single node, this partial flag variety is a Grassmannian. If $\omega_0$ is the longest word in the Weyl group of type $\Delta$, and $\omega_0^K$ is the longest word in the subgroup generated by simple reflections at nodes not in $J$, then \cite{geisskacmoody}*{Lem.~17.2} we have
\[\Sub{\inj{J}}=\Sub{\Pi_{\omega_0^K\omega_0}}.\]
In particular, the categories $\Sub{\inj{J}}$ are Frobenius cluster categories.

If $\Pi$ is the preprojective algebra of Dynkin type $\Delta$, then we have $\fgmod{\Pi}=\Sub{\Pi}=\Sub{\Pi_{\omega_0}}$, where $\omega_0$ is the longest word in the Weyl group of type $\Delta$, so $\fgmod{\Pi}$ is a Frobenius cluster category. The algebra $A$ appearing in \Cref{i3cyeg} is isomorphic to the endomorphism algebra of a cluster-tilting object in $\fgmod{\Pi}$ for $\Pi$ the preprojective algebra of type $\typeA{2}$, and is thus internally $3$-Calabi--Yau by \Cref{ctiltidcy}. Similarly, the algebra $A'$ from \Cref{i3cyeg} is isomorphic to the endomorphism algebra of a cluster-tilting object in $\Sub{\Pi_{s_2s_1s_3s_2}}=\Sub{\inj{2}}$, where $\Pi$ is the preprojective algebra of type $\typeA{3}$ and $\inj{2}$ is the indecomposable injective module with socle at the bivalent vertex $2$, and so $A'$ is also internally $3$-Calabi--Yau.

The category of projective objects of $\Sub{\Pi_\omega}$ is given by $\add{\Pi_\omega}$. Since $\Sub{\Pi_\omega}$ is a $\Hom$-finite Frobenius cluster category, it follows from \Cref{normalform} that $\Sub{\Pi_\omega}\simeq\GP(\Pi_\omega)$. Since $\Pi_w$ has Gorenstein dimension at most $1$ \cite{buancluster}*{Prop.~III.2.2}, we even have $\Sub{\Pi_\omega}=\GP(\Pi_\omega)$ as full subcategories of $\fgmod{\Pi_w}$. Note that the Gorenstein dimension of $\Pi_w$ is strictly smaller than the bound provided by \Cref{normalform}.
\end{eg}

\begin{eg}
\label{jkscats}
Our second family of examples was introduced by Jensen--King--Su \cite{jensencategorification} to categorify the cluster algebra structure on the homogeneous coordinate ring of the Grassmannian $\Grass{k}{n}$ of $k$-planes in $\CC^n$. Each category in this family is of the form $\CM(B)$ for a Gorenstein order $B$ (depending on positive integers $1<k<n$) over $Z=\powser{\CC}{t}$. One description of $B$ is as follows. Let $\Delta$ be the graph (of affine type $\afftypeA{n-1}$) with vertex set given by the cyclic group $\ZZ_n$, and edges between vertices $i$ and $i+1$ for all $i$. Let $\Pi$ be the completion of the preprojective algebra on $\Delta$ with respect to the arrow ideal. Write $x$ for the sum of `clockwise' arrows $i\to i+1$, and $y$ for the sum of `anti-clockwise' arrows $i\to i-1$. Then we have
\[B=\Pi/\Span{x^k-y^{n-k}}.\]
In this description, $Z$ may be identified with the centre $\powser{\CC}{xy}$ of $B$.

Objects of $\CM(B)$ are $B$-modules that are free and finitely generated over $Z$. Since $Z$ is a principal ideal domain, and hence Noetherian, any submodule of a free and finitely generated $Z$-module is also free and finitely generated, and so $\CM(B)$ is closed under subobjects. In particular, $\CM(B)$ is closed under kernels of epimorphisms. Moreover \cite{jensencategorification}*{Cor.~3.7}, $B\in\CM(B)$, and so $\syz{}(\fgmod{B})\subseteq\CM(B)$.

As a $Z$-module, any object $M\in\CM(B)$ is isomorphic to $Z^k$ for some $k$, and so $\Endalg{Z}{M}\iso Z^{k^2}$ is a finitely generated $Z$-module. Since $Z$ is Noetherian, the algebra $\Endalg{B}{M}\subseteq\Endalg{Z}{M}$ is also finitely generated as a $Z$-module. Thus $\Endalg{B}{M}$ is Noetherian, as it is finitely generated as a module over the commutative Noetherian ring $Z$. We may now apply \Cref{gldimbound} to see that any cluster-tilting object $T\in\CM(B)$ satisfies $\gldim{\Endalg{B}{T}}\leq 3$. Moreover \cite{jensencategorification}*{Cor.~4.6}, $\stabCM(B)=\stabSub{\inj{k}}$, where $\inj{k}$ is an indecomposable injective module for the preprojective algebra of type $\typeA{n-1}$ (see \Cref{birscats}), so $\stabCM(B)$ is $2$-Calabi--Yau, and $\CM(B)$ is a Frobenius cluster category. This category is not $\Hom$-finite, unlike the categories $\Sub{\Pi_\omega}$.

The algebra $B$ is $1$-Iwanaga--Gorenstein, so the Gorenstein dimension is again strictly smaller than the bound in \Cref{normalform}, and we have equalities $\CM(B)=\GP(B)=\Sub{B}$ \cite{jensencategorification}*{Cor.~3.7}.

Baur--King--Marsh \cite{baurdimer}*{Thm.~10.3} show that for certain cluster-tilting objects $T\in\CM(B)$, the endomorphism algebra $\Endalg{B}{T}$ is isomorphic to a frozen Jacobian algebra (\Cref{frjacalg}) associated to a dimer model on a disk, with the projection onto a maximal projective-injective summand corresponding to the sum of idempotents at the frozen vertices. By \Cref{ctiltidcy}, these dimer algebras, which satisfy a natural consistency condition \cite{baurdimer}*{\S5}, are internally $3$-Calabi--Yau with respect to their boundary idempotent; cf.\ Broomhead \cite{broomheaddimer}*{\S7}, who shows that consistent dimer models on closed surfaces give rise to $3$-Calabi--Yau Jacobian algebras.
\end{eg}

\begin{eg}
\label{mckay}
The algebraic McKay correspondence provides many examples of Frobenius $1$-cluster categories. Let the special linear group $\SL_2(\CC)$ act on $\powser{\CC}{x,y}$ in the natural way. Let $G$ be a finite subgroup of $\SL_2(\CC)$, and consider the invariant ring $R=\powser{\CC}{x,y}^G$. For example, if $G$ is cyclic of order $n$, generated by
\[\begin{pmatrix}\omega&0\\0&\omega^{-1}\end{pmatrix}\]
for some primitive $n$-th root of unity $\omega$, then $R=\powser{\CC}{x^n,xy,y^n}$.

A well-known result of Herzog \cite{herzogringe} shows that $\powser{\CC}{x,y}$ is an additive generator (or equivalently, a $1$-cluster-tilting object) of the Frobenius category $\CM(R)$ of maximal Cohen--Macaulay $R$-modules. By computing Auslander--Reiten sequences in $\CM(R)$, as in Leuschke--Wiegand  \cite{leuschkecohenmacaulay}*{Prop.~13.22}, one can see that the Auslander--Reiten translation on $\CM(R)$ is trivial, and so $\CM(R)$ is stably $1$-Calabi--Yau. Let $T$ be a basic $R$-module such that $\add_R{T}=\add_R{\powser{\CC}{x,y}}$, so that $T$ is the unique (up to isomorphism) basic $1$-cluster-tilting object of $\CM(R)$. By Auslander's Theorem \cite{auslanderrational} and a result of Reiten--Van den Bergh \cite{reitentwodimensional}, there are isomorphisms
\begin{align*}
\Endalg{R}{T}&\isoto\preproj{\widetilde{\Delta}},\\
\stabEndalg{R}{T}&\isoto\preproj{\Delta},
\end{align*}
where $\widetilde{\Delta}$ is the extended Dynkin diagram given by the McKay graph of $G$, and $\Delta$ is its unextended counterpart. It is well-known that $\preproj{\Delta}$ is finite dimensional, so $\stabCM(R)$ is $\Hom$-finite, and that $\gldim{\preproj{\widetilde{\Delta}}}\leq2$. Thus $\CM(R)$ is a Frobenius $1$-cluster category.

The algebra $\Endalg{R}{T}\iso\preproj{\widetilde{\Delta}}$ is bimodule $2$-Calabi--Yau, which is consistent with (but stronger than) the conclusion of \Cref{ctiltidcy}.
\end{eg}

\section{\texorpdfstring{From internally $d$-Calabi--Yau algebras to $d$-cluster-tilting objects}{From internally d-Calabi--Yau algebras to d-cluster-tilting objects}}
\label{idcytodct}

\Cref{ctiltidcy} shows how internally $(m+1)$-Calabi--Yau algebras arise as endomorphism algebras of cluster-tilting objects in Frobenius $m$-cluster categories. In this section we work in the opposite direction, and show how to construct a Frobenius category admitting a $(d-1)$-cluster-tilting object from the data of an internally $d$-Calabi--Yau algebra, thus generalising a result of Amiot--Iyama--Reiten \cite{amiotstable}*{Thm.~2.3} on bimodule $d$-Calabi--Yau algebras. Since we will work only with Noetherian algebras, by \Cref{normalform} we should expect to produce categories of the form $\GP(B)$ for some Iwanaga--Gorenstein algebra $B$, and indeed this is what we shall do. Our main result is the following.

\begin{thm}[cf.~\cite{amiotstable}*{Thm.~2.3}]
\label{airanalogue}
Let $A$ be a Noetherian algebra and let $e\in A$ be an idempotent such that $A/\Span{e}$ is finite dimensional, and both $A$ and $A^{\op}$ are internally $d$-Calabi--Yau with respect to $e$. Write $B=eAe$ and $\stab{A}=A/\Span{e}$. Then
\begin{itemize}
\item[(i)]$B$ is Iwanaga--Gorenstein with Gorenstein dimension at most $d$, so $\GP(B)$ is a Frobenius category,
\item[(ii)]$eA$ is $(d-1)$-cluster-tilting in $\GP(B)$, and
\item[(iii)]there are natural isomorphisms $\Endalg{B}{eA}\isoto A$ and $\Endalg{\stabGP(B)}{eA}\isoto\stab{A}$.
\end{itemize}
\end{thm}

\begin{rem}
\label{airanalogueop}
While all of the conclusions of \Cref{airanalogue}, except for $B$ being Iwanaga--Gorenstein, refer only to left $B$-modules, the proof we will give uses the assumption that $A^{\op}$ is internally $d$-Calabi--Yau to draw conclusions about right $A$-modules. This applies in particular to showing that the right $A$-module $eA$ is cluster-tilting in the category of Gorenstein projective $B$-modules; see \Cref{ctiltleft,ctiltright}.

Since the assumptions of \Cref{airanalogue} are symmetric in $A$ and $A^{\op}$, we may also conclude that $\GP(B^{\op})$ is a Frobenius category in which $Ae$ is a $(d-1)$-cluster-tilting object, and there are natural isomorphisms $\Endalg{B^{\op}}{Ae}\isoto A^{\op}$ and $\Endalg{\stabGP(B^{\op})}{Ae}\isoto\stab{A}^{\op}$.

We emphasise two cases in which the assumptions of \Cref{airanalogue} may be made to appear one-sided. Firstly, as in \Cref{fdidcy}, if $A$ is a finite dimensional algebra then it is internally $d$-Calabi--Yau with respect to $e$ if and only if the same is true of $A^{\op}$. Secondly, if $A$ is bimodule internally $d$-Calabi--Yau with respect to $e$, then (\Cref{lrsym}) so is $A^{\op}$, and therefore both $A$ and $A^{\op}$ are internally $d$-Calabi--Yau with respect to $e$ by \Cref{bimodtocat}.
\end{rem}

We note that Amiot--Iyama--Reiten's result \cite{amiotstable}*{Thm.~2.3} is a special case of our \Cref{airanalogue}. To obtain the same conclusions, they assume that $A$ is Noetherian, $A/\Span{e}$ is finite dimensional, and that $A$ is bimodule $d$-Calabi--Yau. By \Cref{bimodtocat}, this means that both $A$ and $A^{\op}$ are internally $d$-Calabi--Yau with respect to any idempotent, in particular with respect to $e$.

The rest of the section is largely devoted to proving \Cref{airanalogue}, so we let $A$, $e$, $\stab{A}$ and $B$ be as in the assumptions of this theorem. We begin with the following straightforward observation.

\begin{prop}
\label{Bnoetherian}
The algebra $B$ is Noetherian.
\end{prop}
\begin{proof}
Any left ideal $I$ of $B$ is of the form $e\lift{I}$ for a left ideal $\lift{I}=AI$ of $A$. So any ascending chain of left ideals of $B$ determines and is determined by such a chain of ideals of $A$, which stabilises as $A$ is Noetherian. A similar argument shows that $B$ is right Noetherian.
\end{proof}

\begin{prop}[cf.~\cite{amiotstable}*{Lem.~2.6}]
\label{airextvanish}
For any $X\in\fgmod{\stab{A}}$, we have
\begin{itemize}
\item[(i)]$\Ext^i_A(X,A)=0$ for $i\ne d$, and
\item[(ii)]$\Ext^i_A(X,Ae)=0$ for any $i\in\ZZ$.
\end{itemize}
\end{prop}
\begin{proof}
Both $A$ and $Ae$ are finitely generated projective $A$-modules, and so are in particular perfect. Since $\stab{A}$, and therefore $X$, is finite dimensional, $X$ is a finitely generated $A$-module, and thus perfect since $A$ is Noetherian of finite global dimension. Now we can use the internal Calabi--Yau duality of $A$ to deduce that
\[\Ext^i_A(X,A)=\Kdual{\Ext^{d-i}_A(A,X)}=0\]
and
\[\Ext^i_A(X,Ae)=\Kdual\Ext^{d-i}_A(Ae,X)=0\]
for $i\ne d$, since $A$ and $Ae$ are projective. Since $X\in\fgmod{\stab{A}}$, we have $eX=0$, and so
\[\Ext^d_A(X,Ae)=\Kdual{\Hom_A(Ae,X)}=\Kdual(eX)=0.\qedhere\]
\end{proof}
The assumption of part (i) of \Cref{airextvanish} is slightly more restrictive than that of \cite{amiotstable}*{Lem.~2.6(a)}. This is necessary for the result to hold in our setting, since our $A$ is only internally $d$-Calabi--Yau. However, this stronger assumption is satisfied whenever \cite{amiotstable}*{Lem.~2.6(a)} is used in the proof of \cite{amiotstable}*{Thm.~2.3}.

The following results (\Cref{extcalcs,pdimbound,ctiltleft,ctiltright}) are now close analogues of \cite{amiotstable}*{Prop.~2.7, Lem.~2.9--2.11}, with very similar proofs. For the convenience of the reader, and to make the paper more self-contained, we give a complete argument using our notation and conventions.

\begin{prop}[cf.~\cite{amiotstable}*{Prop.~2.7}]
\label{extcalcs}
We have isomorphisms
\[\Ext^i_B(eA,B)\iso\begin{cases}Ae,&i=0,\\0,&i\ne0\end{cases}\]
of $A\tens_\KK B^{\op}$-modules, and isomorphisms
\[\Ext^i_B(eA,eA)\iso\begin{cases}A^{\op},&i=0,\\0,&0<i<d-1,\end{cases}\]
of vector spaces, the isomorphism in case $i=0$ being additionally an isomorphism of algebras.
\end{prop}
\begin{proof}
We can compute $\Ext^i_B(eA,B)$ as the cohomology of
\[\RHom_B(eA,B)\iso\RHom_B(eA,\RHom_A(Ae,Ae))\iso\RHom_A(Ae\ltens_BeA,Ae),\]
and wish to show that this is isomorphic to the cohomology of $\RHom_A(A,Ae)$. To do this, we show that
\[\RHom_A(Ae\ltens_BeA,Ae)\iso\RHom_A(A,Ae).\]

Let $f$ be the composition of the natural map
\[Ae\ltens_BeA\to\cohom{0}{Ae\ltens_BeA}=Ae\tens_BeA\]
with the multiplication map $Ae\tens_BeA\to A$, and let $X$ be the mapping cone of $f$, so we have a triangle
\begin{equation}
\label{magictriangle}
\begin{tikzcd}[column sep=20pt]
Ae\ltens_BeA\arrow{r}{f}&A\arrow{r}&X\arrow{r}&Ae\ltens_BeA[1]
\end{tikzcd}
\end{equation}
in $\dcat{A}$. The map $eA\tens_Af$ is the natural isomorphism $B\ltens_BeA\isoto eA$, so $eA\tens_AX=0$. It follows that $e\cohom{i}{X}=0$, and hence $\cohom{i}{X}\in\fgmod{\stab{A}}$ for all $i\in\ZZ$. Thus, by \Cref{airextvanish}, $\Ext^j_A(\cohom{i}{X},Ae)=0$ for all $i,j\in\ZZ$.

We can compute $\cohom{k}{\RHom_A(X,Ae)}$ via a hypercohomology spectral sequence ${}^{II}E^{ij}_r$ \cite{weibelintroduction}*{\S5.7.9, see also Defn.~5.6.2}, in which
\[{}^{II}E^{ij}_2=\Ext^j_A(\cohom{i}{X},Ae)=0\]
as above. It follows that $\cohom{k}{\RHom_A(X,Ae)}=0$ for all $k$, and so $\RHom_A(X,Ae)=0$. Now applying $\RHom_A(-,Ae)$ to the triangle (\ref{magictriangle}) yields the required isomorphism
\[\RHom_A(Ae\ltens_BeA,Ae)\iso\RHom_A(A,Ae)\]
in $\dcat{A\tens_\KK B^{\op}}$, from which the first assertion follows by our initial calculations.

Similarly, we have isomorphisms
\[\RHom_B(eA,eA)\iso\RHom_B(eA,\RHom_A(Ae,A))\iso\RHom_A(Ae\ltens_BeA,A),\]
and so to obtain the second assertion we wish to show that
\[\RHom_A(Ae\ltens_BeA,A)\iso\RHom_A(A,A).\]
Both $Ae$ and $eA$ are concentrated in degree $0$, so by triangle (\ref{magictriangle}) we have $\cohom{i}{Ae\ltens_BeA}=0$ for $i>0$, and so $\cohom{i}{X}=0$ for $i>0$. Recalling that $\cohom{i}{X}\in\fgmod{\stab{A}}$, it follows from \Cref{airextvanish} that $\Ext^j_A(\cohom{i}{X},A)=\Hom_{\dcat{A}}(X,A[i])=0$ for $j\ne d$. By an analogous spectral sequence argument to above, $\cohom{i}{\RHom_A(X,A)}=0$ for $i<d$.

From (\ref{magictriangle}), we obtain the long exact sequence
\[\begin{tikzcd}[column sep=13pt]
\cdots\arrow{r}&\Hom_{\dcat{A}}(X,A[i])\arrow{r}&\Hom_{\dcat{A}}(A,A[i])\arrow{r}&\Hom_{\dcat{A}}(Ae\ltens_BeA,A[i])\arrow{r}&\cdots
\end{tikzcd}\]
As $\Hom_{\dcat{A}}(X,A[i])=0$ for $i<d$ as above, it follows from our initial calculations that
\[\Ext^i_B(eA,eA)\iso\Hom_{\dcat{A}}(Ae\ltens_BeA,A[i])\iso\Hom_{\dcat{A}}(A,A[i])\iso\begin{cases}A^{\op},&i=0,\\0,&0<i<d-1,\end{cases}\]
as required.
\end{proof}

\begin{lem}[cf.~\cite{amiotstable}*{Lem.~2.9}]
\label{pdimbound}
For any $X\in\fgmod{B}$, we have
\[\pdim_{A^{\op}}\Hom_B(X,eA)\leq d-2.\]
\end{lem}
\begin{proof}
Since $B$ is Noetherian by \Cref{Bnoetherian}, $X$ has a projective presentation $P_1\to P_0\to X\to 0$ with $P_0$ and $P_1$ finitely generated. Applying $\Hom_B(-,eA)$ gives the exact sequence
\[\begin{tikzcd}[column sep=20pt]0\arrow{r}&\Hom_B(X,eA)\arrow{r}&\Hom_B(P_0,eA)\arrow{r}&\Hom_B(P_1,eA)
\end{tikzcd}\]
of $A^{\op}$-modules. Since $\Hom_B(P_i,eA)$ is a projective $A^{\op}$-module, the above sequence shows that $\Hom_B(X,eA)$ is a second syzygy module. Since $\gldim{A^{\op}}\leq d$ by the assumption that $A^{\op}$ is internally $d$-Calabi--Yau, it follows that $\pdim_{A^{\op}}\Hom_B(X,eA)\leq d-2$.
\end{proof}

\begin{rem}
We can obtain the statement that $\gldim{A^{\op}}\leq d$ needed in the proof of \Cref{pdimbound} without assuming that $A^{\op}$ is internally $d$-Calabi--Yau. By Noetherianity of $A$, we have $\gldim{A^{\op}}=\gldim{A}$, and $\gldim{A}\leq d$ since $A$ is internally $d$-Calabi--Yau. However, the next two results, \Cref{ctiltleft,ctiltright}, will use this assumption on $A^{\op}$ in a more fundamental way.
\end{rem}

\begin{lem}[cf.~\cite{amiotstable}*{Lem.~2.10}]
\label{ctiltleft}
If $X\in\GP(B)$ and $\Ext^i_B(X,eA)=0$ for all $0<i<d-1$, then $X\in\add{{}_B(eA)}$.
\end{lem}
\begin{proof}
Pick an exact sequence
\[\begin{tikzcd}[column sep=20pt]
0\arrow{r}&Y\arrow{r}&P_{d-3}\arrow{r}&\cdots\arrow{r}&P_0\arrow{r}&X\arrow{r}&0
\end{tikzcd}\]
in which each $P_i$ is a finitely generated projective $B$-module. By the assumption on the vanishing of $\Ext^i_B(X,eA)$, we can apply $\Hom_B(-,eA)$ to obtain an exact sequence
\[\begin{tikzcd}[column sep=9pt]
0\arrow{r}&\Hom_B(X,eA)\arrow{r}&\Hom_B(P_0,eA)\arrow{r}&\cdots\arrow{r}&
\Hom_B(P_{d-3},eA)\arrow{r}&\Hom_B(Y,eA)\arrow{r}&0
\end{tikzcd}\]
of $A^{\op}$-modules. Each $\Hom_B(P_i,eA)$ is a projective $A^{\op}$-module, and by \Cref{pdimbound} we have $\pdim_{A^{\op}}\Hom_B(Y,eA)\leq d-2$, so $\Hom_B(X,eA)$ is also a projective $A^{\op}$-module. It follows that $\Hom_B(X,B)=\Hom_B(X,eA)e\in\add_{B^{\op}}(Ae)$. By \cite{amiotstable}*{Prop.~1.3(b)} there are quasi-inverse dualities
\begin{align*}
\Hom_B(-,B)&\colon\GP(B)\to\GP(B^{\op}),\\
\Hom_{B^{\op}}(-,B)&\colon\GP(B^{\op})\to\GP(B).
\end{align*}
Since we are assuming $A^{\op}$ is also internally $d$-Calabi--Yau with respect to $e$, we can apply \Cref{extcalcs} to $A^{\op}$ to obtain an isomorphism $\Hom_{B^{\op}}(Ae,B)\isoto eA$ of $B$-modules. Therefore
\[X\iso\Hom_{B^{\op}}(\Hom_B(X,B),B)\in\add{{}_B(\Hom_{B^{\op}}(Ae,B))}=\add{{}_B(eA)}\]
as required.
\end{proof}

\begin{lem}[cf.~\cite{amiotstable}*{Lem.~2.11}]
\label{ctiltright}
If $X\in\GP(B)$ and $\Ext^i_B(eA,X)=0$ for all $0<i<d-1$, then $X\in\add{{}_B(eA)}$.
\end{lem}
\begin{proof}
The quasi-inverse dualities
\begin{align*}
\Hom_B(-,B)&\colon\GP(B)\to\GP(B^{\op}),\\
\Hom_{B^{\op}}(-,B)&\colon\GP(B^{\op})\to\GP(B)
\end{align*}
from \cite{amiotstable}*{Prop.~1.3(b)} preserve extension groups. Since $\Hom_B(eA,B)\iso Ae$ by \Cref{extcalcs}, it follows that $\Ext^i_{B^{\op}}(\Hom_B(X,B),Ae)=0$ for all $0<i<d-1$. Thus by applying \Cref{ctiltleft} to $A^{\op}$ and $\Hom_B(X,B)\in\GP(A^{\op})$, we find that $\Hom_B(X,B)\in\add{{}_{B^{\op}}(Ae)}$. Then, as in \Cref{ctiltleft}, applying $\Hom_{B^{\op}}(-,B)$ gives $X\in\add{{}_B(eA)}$.
\end{proof}

We are now ready to prove \Cref{airanalogue}.

\begin{proof}[Proof of \Cref{airanalogue}]
\begin{itemize}
\item[(i)]We have already shown in \Cref{Bnoetherian} that $B$ is Noetherian, so it remains to show that $B$ has injective dimension at most $d$ on each side. First we show that $\Ext^{d+1}_B(X,B)=0$ for all $X\in\fgmod{B}$. Given such an $X$, let $Y=Ae\tens_BX$, and let $\bP$ be a projective resolution of $Y$. Then $e\bP$ is a bounded complex in the full subcategory $\add(eA)$ of $\fgmod{B}$, quasi-isomorphic to $eY=X$. By \Cref{extcalcs}, $\Ext^i_B(eA,B)=0$ for $i>0$, so another spectral sequence argument (now using ${}^IE_{pq}^r$ from \cite{weibelintroduction}*{Defn.~5.6.1}) shows that
\[\Ext^{d+1}_B(X,B)\iso\cohom{d+1}{\Hom_B(e\bP,B)},\]
where $\Hom_B(e\bP,B)$ denotes the complex obtained by applying $\Hom_B(-,B)$ to $e\bP$.
Since
\begin{align*}
\Hom_B(e\bP,B)&=\Hom_B(eA\tens_A\bP,B)\\
&=\Hom_A(\bP,\Hom_B(eA,B))\iso\Hom_A(\bP,Ae),
\end{align*}
with the final isomorphism coming from \Cref{extcalcs}, it follows that
\[\Ext^{d+1}_B(X,B)\iso\cohom{d+1}{\Hom_A(\bP,Ae)}\iso\Ext^{d+1}_A(Y,Ae)=0\]
since $\gldim{A}\leq d$ by assumption. A dual argument, using that $A^{\op}$ is internally $d$-Calabi--Yau with respect to $e$, shows that $\Ext^{d+1}_{B^{\op}}(X,B)=0$ for all $X\in\fgmod{B^{\op}}$. It follows that $B$ is Iwanaga--Gorenstein of dimension at most $d$, and so $\GP(B)$ is Frobenius \cite{buchweitzmaximal}*{\S4.8}.

\item[(ii)]Since $A$ is Noetherian, the left ideal $\Span{e}=AeA$ is finitely generated. Thus it has a finite generating set contained in $eA$, which must generate $eA\subseteq AeA$ as a $B$-module, so $eA\in\fgmod{B}$. Now $eA\in\GP(B)$ and $\Ext^i_B(eA,eA)=0$ for $0<i<d-1$ by \Cref{extcalcs}. This, together with \Cref{ctiltleft,ctiltright}, shows that $eA$ is $(d-1)$-cluster-tilting in $\GP(B)$.

\item[(iii)]We have $\Endalg{B}{eA}\iso A$ by \Cref{extcalcs}, and thus we have an equivalence
\[\Hom_B(eA,-)\colon\add{{}_B(eA)}\isoto\add{{}_AA}.\]
By \Cref{extcalcs} again, $\Hom_B(eA,B)\iso Ae$. It follows that
\begin{align*}
\Endalg{\stabGP(B)}{eA}&=\Endalg{B}{eA}/\Span{\add{{}_BB}}\\
&\iso\Endalg{A}{A}/\Span{\add{{}_A(Ae)}}\iso A/\Span{e}=\stab{A}
\end{align*}
where $\Span{\cat}$ denotes the ideal of maps factoring through the subcategory $\cat$.
\qedhere
\end{itemize}
\end{proof}

In the setting of \Cref{airanalogue}, we would also like to conclude that $\GP(B)$ is stably $(d-1)$-Calabi--Yau. We now show, using descriptions by Kalck--Yang \cite{kalckrelative} of $\stabGP(B)$ in terms of (complexes of) $A$-modules, that this category is $(d-1)$-Calabi--Yau when we strengthen the assumptions of Theorem~\ref{airanalogue} to require that $A$ is bimodule internally $d$-Calabi--Yau.

\begin{thm}
\label{bimodicytostabcy}
Let $A$ be a Noetherian algebra and let $e\in A$ be an idempotent such that $A/\Span{e}$ is finite dimensional, and $A$ is bimodule internally $d$-Calabi--Yau with respect to $e$. Write $B=eAe$. Then all of the conclusions of \Cref{airanalogue} hold, and moreover $\stabGP(B)$ is $(d-1)$-Calabi--Yau.
\end{thm}
\begin{proof}
By \Cref{bimodtocat}, $A$ and $A^{\op}$ are internally $d$-Calabi--Yau with respect to $e$, so our assumptions imply those of \Cref{airanalogue}. It remains to check that $\stabGP(B)$ is $(d-1)$-Calabi--Yau.

By \cite{kalckrelative}*{Prop.~2.10} (see also \cite{dwyernoncommutative}*{\S2--3}), there exists a dg-algebra $C$ and a dg-algebra homomorphism $A\to C$, where $A$ is considered as a dg-algebra concentrated in degree $0$, such that $C$ fits into a recollement
\[\begin{tikzcd}[column sep=20pt]
\dcat{C}\arrow{r}&\dcat{A}\arrow[bend right]{l}\arrow[bend left]{l}\arrow{r}&\dcat{B}.\arrow[bend right]{l}\arrow[bend left]{l}
\end{tikzcd}\]
Moreover, $C^i=0$ for $i>0$ and $\cohom{0}{C}=A/\Span{e}=\stab{A}$. Since $\stab{A}$ is finite dimensional and $\gldim{A}\leq d$, it follows from \cite{kalckrelative}*{Cor.~2.13} that $\dim\cohom{i}{C}<\infty$ for all $i$, and so $\per{C}$ is $\Hom$-finite. The proof of this corollary also shows that $\dcatfd{C}\subseteq\per{C}$.

By \cite{kalckrelative}*{Proof of Cor.~2.12}, the functor $i_*=\RHom_C(C,-)\colon\dcat{C}\to\dcat{A}$ induces a triangle equivalence $\dcat{C}\isoto\resdcat{\stab{A}}{A}$, which restricts to a triangle equivalence $\dcatfd{C}\isoto\resdcatfd{\stab{A}}{A}$. Thus for any $M\in\dcatfd{C}$ and $N\in\per{C}$, we have functorial isomorphisms
\[\Kdual\Hom_{\dcat{C}}(M,N)=\Kdual\Hom_{\dcat{A}}(i_*M,i_*N)=\Hom_{\dcat{A}}(i_*N,i_*M[d])=\Hom_{\dcat{C}}(N,M[d]),\]
the second coming from \Cref{bimodcydual}, using the assumption that $A$ is bimodule internally $d$-Calabi--Yau with respect to $e$. Thus $(\per{C},\dcatfd{C},\add{C})$ is a $d$-Calabi--Yau triple in the sense of Iyama--Yang \cite{iyamasilting}*{\S5.1}, and so it follows from \cite{iyamasilting}*{Thm.~5.8(a)} (see also \cite{amiotcluster}*{\S1}, \cite{guocluster}*{\S2}) that $\per{C}/\dcatfd{C}$ is $(d-1)$-Calabi--Yau.

We complete the proof by showing that $\stabGP(B)$ is equivalent to a full triangulated subcategory of $\per{C}/\dcatfd{C}$, and so is also $(d-1)$-Calabi--Yau. Since $\gldim{A}\leq d$, \cite{kalckrelative}*{Cor.~2.12a} tells us that $i^*=C\ltens_A-$ provides a triangle equivalence between the idempotent completion of $\bdcat{A}/\thick{Ae}$, denoted by $(\bdcat{A}/\thick{Ae})^\omega$, and $\per{C}$. Moreover, \cite{kalckrelative}*{Proof of Cor.~2.12} shows that $i_*=\RHom_C(C,-)$ induces a triangle equivalence $\dcatfd{C}\isoto\resdcatfd{\stab{A}}{A}$, and that the codomain of this equivalence coincides with $\thick(\fgmod{\stab{A}})$. Since $i^*i_*\simeq\id{\dcat{C}}$, we see that $i^*$ restricts to an equivalence $\thick(\fgmod{\stab{A}})\isoto\dcatfd{C}$, and so induces an equivalence
\[\frac{(\bdcat{A}/\thick{Ae})^\omega}{q(\thick(\fgmod{\stab{A}}))}\isoto\frac{\per{C}}{\dcatfd{C}},\]
where $q$ denotes the projection $\bdcat{A}\to\bdcat{A}/\thick{Ae}$, which restricts to an equivalence on $\thick(\fgmod{\stab{A}})$ by the above observations. We also have equivalences
\[\frac{\bdcat{A}/\thick{Ae}}{q(\thick(\fgmod{\stab{A}}))}\isoto\bdcat{B}/\per{B}\isoto\stabGP(B),\]
the first from \cite{kalckrelative}*{Prop.~3.3} and the second from a result of Buchweitz \cite{buchweitzmaximal}*{Thm.~4.4.1}. Since $\bdcat{A}/\thick{Ae}$ is a full triangulated subcategory of its idempotent completion, combining the above equivalences gives a triangle equivalence of $\stabGP(B)$ with a full triangulated subcategory of the $(d-1)$-Calabi--Yau triangulated category $\per{C}/\dcatfd{C}$, from which it follows that $\stabGP(B)$ is itself $(d-1)$-Calabi--Yau.
\end{proof}

It was necessary in the proof of \Cref{bimodicytostabcy} to use the bimodule internal Calabi--Yau symmetry of $A$ to obtain a duality between spaces of maps of complexes of $A$-modules, so it is unclear whether the conclusion that $\stabGP(B)$ is $(d-1)$-Calabi--Yau might hold only under the weaker assumptions of Theorem~\ref{airanalogue}. As already stated, we do not currently have any examples of internally Calabi--Yau algebras that are not bimodule internally Calabi--Yau, but it seems unlikely that the two classes coincide.

Under the assumptions of \Cref{bimodicytostabcy}, we would like to be able to conclude that the Frobenius category $\GP(B)$ is in fact a Frobenius $(d-1)$-cluster category in the sense of Definition~\ref{frobclustcat}. It remains to check that $\gldim{\Endalg{B}{T}}\leq d$ for any $(d-1)$-cluster-tilting object $T\in\GP(B)$; a priori, we only know this for the $(d-1)$-cluster-tilting object $eA$. Whenever $\Endalg{B}{T}$ is Noetherian, we can apply \Cref{gldimbound} to get the desired conclusion. While this Noetherianity is automatic in some situations, such as if $B$ is finite dimensional over $\KK$, in general it appears to be a more subtle issue.

\begin{rem}
\label{assumptionsrem}
As indicated in the introduction, \Cref{airanalogue,bimodicytostabcy} are motivated by the problem of constructing Frobenius categories modelling cluster algebras. Given the seed of a cluster algebra with frozen variables, we can look for an algebra $A$ with the same quiver (up to the addition of arrows between frozen vertices), satisfying the conditions of \Cref{bimodicytostabcy} for $d=3$, and then apply this theorem to obtain the Frobenius category $\GP(B)$. Since constructing such an $A$ can be very difficult, we wish to comment on the degree to which the conditions we are imposing are necessary.

Firstly, we consider it likely that the condition that $A$ is Noetherian can be dropped, up to finding an appropriate replacement for the category $\GP(B)$. While $B$ may not be Noetherian if $A$ fails to be, it will still have injective dimension at most $d$ on each side, so there should be a `good' theory of Gorenstein projective modules over $B$. For our methods to work, we would need the analogues of \cite{amiotstable}*{Prop.~1.3} to hold in this setting. We would also hope for a more general version of the Iyama--Kalck--Wemyss--Yang equivalence stated here as \Cref{kiwythm}, without the Noetherianity assumption, which would then apply to arbitrary Frobenius $m$-cluster categories, and Buchweitz's description of the stable category \cite{buchweitzmaximal}*{Thm.~4.4.1}.

The other conditions are more essential; if $A=\Endalg{\frobcat}{T}$ for $T$ a cluster-tilting object in a Frobenius cluster category $\frobcat$, and $e$ is the idempotent given by projecting onto a maximal projective summand of $T$, then $A/\Span{e}$ must be finite dimensional since $\stab{\frobcat}$ is Hom-finite, and $A$ is internally $3$-Calabi--Yau with respect to $e$ by \Cref{ctiltidcy}. On the other hand, it may not be necessary for $A$ to be bimodule internally $3$-Calabi--Yau.
\end{rem}

\section{A bimodule complex for frozen Jacobian algebras}
\label{bimodcomplex}

Given a Frobenius cluster category $\frobcat$ and a cluster-tilting object $T\in\frobcat$, it is often the case that $A=\Endalg{\frobcat}{T}$ takes the form of a frozen Jacobian algebra (see \Cref{frjacalg} below). Indeed, this is the case for at least some cluster-tilting objects in the families of Frobenius cluster categories we described in \Cref{birscats,jkscats}; see \cite{buanmutation}*{Thm.~6.6}, \cite{baurdimer}*{Thm.~10.3}. Thus these algebras, which also come with a preferred `frozen' idempotent, are ideal candidates for constructing stably $2$-Calabi--Yau Frobenius categories via the methods of \Cref{airanalogue} and \Cref{bimodicytostabcy}. Moreover, $3$-Calabi--Yau properties of ordinary Jacobian algebras have been widely studied, for example by Bocklandt \cite{bocklandtgraded} and, in the context of dimer models on closed surfaces, by Broomhead \cite{broomheaddimer}.

With this in mind, the main result of this section, \Cref{frjaci3cy}, shows that a frozen Jacobian algebra admitting a particular bimodule resolution (analogous to one defined by Ginzburg \cite{ginzburgcalabiyau}*{5.1.5}) is bimodule internally $3$-Calabi--Yau with respect to its frozen idempotent.

\begin{defn}[cf.~\cite{buanmutation}*{Defn.~1.1}, \cite{demoneticequivers1}*{\S2.1}, \cite{francobipartite}*{\S6.1}]
\label{frjacalg}
An \emph{ice quiver} $(Q,F)$ consists of a finite quiver $Q$ without loops and a (not necessarily full) subquiver $F$ of $Q$. Denote by $\comp{\KK Q}$ the completion of the path algebra of $Q$ over $\KK$ with respect to the arrow ideal. A \emph{potential} on $Q$ is a linear combination $W$ of cycles of $Q$. A vertex or arrow of $Q$ is called \emph{frozen} if it is a vertex or arrow of $F$, and \emph{mutable} or \emph{unfrozen} otherwise. For brevity, we write $Q_0^\mut=Q_0\setminus F_0$ and $Q_1^\mut=Q_1\setminus F_1$ for the sets of mutable vertices and unfrozen arrows respectively. For $\alpha\in Q_1$ and $\alpha_n\dotsm\alpha_1$ a cycle in $Q$, write
\[\der{\alpha}{\alpha_n\dotsm\alpha_1}=\sum_{\alpha_i=\alpha}\alpha_{i-1}\dotsm\alpha_1\alpha_n\dotsm\alpha_{i+1}\]
and extend linearly. The ideal $\Span{\der{\alpha}{W}:\alpha\in Q_1^\mut}$ of $\comp{\KK Q}$ is called the \emph{Jacobian ideal}, and we may take its closure $\close{\Span{\der{\alpha}{W}:\alpha\in Q_1^\mut}}$ since $\comp{\KK Q}$ is a topological algebra. We define the \emph{frozen Jacobian algebra} associated to $(Q,F,W)$ by
\[\frjac{Q}{F}{W}=\comp{\KK Q}/\close{\Span{\der{\alpha}{W}:\alpha\in Q_1^\mut}}.\]
Write $A=\frjac{Q}{F}{W}$. The above presentation of $A$ suggests a preferred idempotent $e=\sum_{v\in F_0}\idemp{v}$, which we call the \emph{frozen idempotent}. We will call $B=eAe$ the \emph{boundary algebra} of $A$.
\end{defn}

\begin{rem}
If $F=\emptyset$, then $\frjac{Q}{\emptyset}{W}=:\jac{Q}{W}$ is the usual Jacobian algebra.
\end{rem}

\begin{eg}
\label{frjacalgeg}
Consider the ice quiver with potential $(Q,F,W)$, where
\[Q=\mathord{\begin{tikzpicture}[baseline=0]
\node at (0,0.5) (1) {$\boxed{1}$};
\node at (2,0.5) (2) {$\boxed{2}$};
\node at (1,-0.5) (3) {$3$};
\path[-angle 90,font=\scriptsize]
	(2) edge node[below right] {$\alpha_2$} (3)
	(3) edge node[below left] {$\alpha_3$} (1);
\path[\frozen,-angle 90,font=\scriptsize]
	(1) edge node[above] {$\alpha_1$} (2);
\end{tikzpicture}}\]
the frozen subquiver $F$ is the full subquiver on vertices $1$ and $2$, indicated by boxed vertices and a dashed arrow, and $W=\alpha_3\alpha_2\alpha_1$. Then the frozen Jacobian algebra $\frjac{Q}{F}{W}$ is the quotient of $\comp{\KK Q}$ by the relations $\der{\alpha_2}W=\alpha_1\alpha_3$ and $\der{\alpha_3}W=\alpha_2\alpha_1$, so $\frjac{Q}{F}{W}$ is the algebra $A$ from \Cref{i3cyeg}. (In this case $A$ is finite dimensional, so it agrees with the quotient of the ordinary path algebra $\KK Q$ by the same relations.) It will follow from \Cref{frjaci3cy} below that $\frjac{Q}{F}{W}$ is bimodule internally $3$-Calabi--Yau with respect to the idempotent $e_1+e_2$ given by summing the idempotents corresponding to frozen vertices. The usual Jacobian algebra $\jac{Q}{W}$ has the additional relation $\alpha_3\alpha_2=0$ and is not bimodule $3$-Calabi--Yau; indeed, it has infinite global dimension.
\end{eg}

Given a quiver with potential $(Q,W)$, Ginzburg \cite{ginzburgcalabiyau}*{5.1.5} (see also \cite{broomheaddimer}*{\S7}) defines a complex of projective bimodules over the associated Jacobian algebra. For $(Q,W)$ a quiver with potential determined by a dimer model on a torus, Broomhead shows in \cite{broomheaddimer}*{Thm.~7.7} that if the dimer model is consistent, then this complex is isomorphic to $A=\jac{Q}{W}$ in $\bdcat{\env{A}}$, and thus provides a projective bimodule resolution of $A$. It follows in this case that $A$ is $3$-Calabi--Yau, with this property arising from a natural symmetry in the bimodule resolution.

We will now define an analogous complex $\res{A}$ for a frozen Jacobian algebra $A=\frjac{Q}{F}{W}$. Our main result (\Cref{frjaci3cy}) will be that if $\res{A}$ is isomorphic to $A$ in $\bdcat{\env{A}}$, then $A$ is bimodule internally $3$-Calabi--Yau with respect to the frozen idempotent $e=\sum_{v\in F_0}\idemp{v}$, in the sense of \Cref{bimodidcy}. While we will write $\res{A}$ for this complex in order to save space, the definition depends not only on $A$ but on the ice quiver with potential $(Q,F,W)$ giving the presentation of $A$ as $\frjac{Q}{F}{W}$.

Recall that we write $Q_0^\mut=Q_0\setminus F_0$ for the set of mutable vertices and $Q_1^\mut=Q_1\setminus F_1$ for the set of unfrozen arrows. We also write $\vout{v}$ for the set of arrows with tail at $v$, and $\vin{v}$ for the set of arrows with head at $v$. Denote the arrow ideal of $A$ by $\rad{A}$, and let $S=A/\rad{A}$. For the remainder of this section, we write $\tens=\tens_S$.

Introduce formal symbols $\rho_\alpha$ for each $\alpha\in Q_1$ and $\omega_v$ for each $v\in Q_0$, and define $S$-bimodule structures on the vector spaces
\begin{align*}
\KK Q_0&=\bigdsum_{v\in Q_0}\KK\idemp{v},&\KK Q_0^\mut&=\bigdsum_{v\in Q_0^\mut}\KK\idemp{v},&\KK F_0&=\bigdsum_{v\in F_0}\KK\idemp{v},\\
\KK Q_1&=\bigdsum_{\alpha\in Q_1}\KK\alpha,&\KK Q_1^\mut&=\bigdsum_{\alpha\in Q_1^\mut}\KK\alpha,&\KK F_1&=\bigdsum_{\alpha\in F_1}\KK\alpha,\\
\KK Q_2&=\bigdsum_{\alpha\in Q_1}\KK\rel{\alpha},&\KK Q_2^\mut&=\bigdsum_{\alpha\in Q_1^\mut}\KK\rel{\alpha},&\KK F_2&=\bigdsum_{\alpha\in F_1}\KK\rel{\alpha},\\
\KK Q_3&=\bigdsum_{v\in Q_0}\KK\omega_v,&\KK Q_3^\mut&=\bigdsum_{v\in Q_0^\mut}\KK\omega_v,&\KK F_3&=\bigdsum_{v\in F_0}\KK\omega_v,
\end{align*}
via the formulae
\begin{align*}
\idemp{v}\cdot\idemp{v}\cdot\idemp{v}&=\idemp{v},\\
\idemp{\head{\alpha}}\cdot\alpha\cdot\idemp{\tail{\alpha}}&=\alpha,\\
\idemp{\tail{\alpha}}\cdot\rel{\alpha}\cdot\idemp{\head{\alpha}}&=\rel{\alpha},\\
\idemp{v}\cdot\omega_v\cdot\idemp{v}&=\omega_v,
\end{align*}
where $\head{\alpha}$ and $\tail{\alpha}$ denote the head and tail of the arrow $\alpha$. For each $i$, the $S$-bimodule $\KK Q_i$ splits as the direct sum
\[\KK Q_i=\KK Q_i^\mut\dsum\KK F_i.\]
Since $\KK Q_0\cong S$, the $A$-bimodule $A\tens \KK Q_0\tens A$ is canonically isomorphic to $A\tens A$, and we will use the two descriptions interchangeably.

We define maps $\bar{\mu}_i\colon A\tens \KK Q_i\tens A\to A\tens \KK Q_{i-1}\tens A$ for $1\leq i\leq 3$. The map $\bar{\mu}_1$ is defined by
\[\bar{\mu}_1(x\tens\alpha\tens y)=x\tens\idemp{\head{\alpha}}\tens\alpha y-x\alpha\tens\idemp{\tail{\alpha}}\tens y,\]
or, composing with the natural isomorphism $A\tens\KK Q_0\tens A\isoto A\tens A$, by
\[\bar{\mu}_1(x\tens\alpha\tens y)=x\tens\alpha y-x\alpha\tens y.\]
For any path $p=\alpha_m\dotsm\alpha_1$ of $Q$, we may define
\[\Delta_\alpha(p)=\sum_{\alpha_i=\alpha}\alpha_m\dotsm\alpha_{i+1}\tens\alpha_i\tens\alpha_{i-1}\dotsm\alpha_1,\]
and extend by linearity to obtain a map $\Delta_\alpha\colon\comp{\KK Q}\to A\tens \KK Q_1\tens A$. We then define
\[\bar{\mu}_2(x\tens\rel{\alpha}\tens y)=\sum_{\beta\in Q_1}x\Delta_\beta(\der{\alpha}{W})y.\]
Finally, let
\[\bar{\mu}_3(x\tens\omega_v\tens y)=\sum_{\alpha\in\vout{v}}x\tens\rel{\alpha}\tens\alpha y-\sum_{\beta\in\vin{v}}x\beta\tens\rel{\beta}\tens y.\]

\begin{defn}
\label{frjacres}
For $A=\frjac{Q}{F}{W}$, let $\res{A}$ be the sequence
\[\begin{tikzcd}[column sep=20pt]
A\tens \KK Q_3^\mut\tens A\arrow{r}{\mu_3}&A\tens \KK Q_2^\mut\tens A\arrow{r}{\mu_2}&A\tens \KK Q_1\tens A\arrow{r}{\mu_1}&A\tens\KK Q_0\tens A
\end{tikzcd}\]
of $A$-bimodules, where $\mu_1=\bar{\mu}_1$, and the maps $\mu_2$ and $\mu_3$ are obtained by restricting $\bar{\mu}_2$ and $\bar{\mu}_3$ to $A\tens \KK Q_2^\mut\tens A$ and $A\tens \KK Q_3^\mut\tens A$ respectively. As $\vout{v}\union\vin{v}\subseteq Q_1^\mut$ for any $v\in Q_0^\mut$, the map $\mu_3$ takes values in $A\tens \KK Q_2^\mut\tens A$ as claimed.
\end{defn}

If $F=\emptyset$, then $\res{A}$ is the complex associated to $(Q,W)$ by Ginzburg \cite{ginzburgcalabiyau}*{5.1.5} and Broomhead \cite{broomheaddimer}*{\S7}. In the general case, $\res{A}$ has already appeared in work of Amiot--Reiten--Todorov \cite{amiotubiquity}*{Proof of Prop.~2.2}.

\begin{lem}
\label{frjacresiscomp}
For a frozen Jacobian algebra $A=\frjac{Q}{F}{W}$, the sequence $\res{A}$ in \Cref{frjacres} is a complex of finitely generated projective $A$-bimodules, and there is a morphism
\[\begin{tikzcd}[column sep=20pt]
A\tens \KK Q_3^\mut\tens A\arrow{r}{\mu_3}\arrow{d}&A\tens \KK Q_2^\mut\tens A\arrow{r}{\mu_2}\arrow{d}&A\tens \KK Q_1\tens A\arrow{r}{\mu_1}\arrow{d}&A\tens A\arrow{d}{\mu_0}\\
0\arrow{r}&0\arrow{r}&0\arrow{r}&A
\end{tikzcd}\]
from $\res{A}$ to $A$, where $\mu_0\colon A\tens A\to A$ is the multiplication in $A$. Moreover, the complex $0\to\res{A}\to A\to0$ is exact at $A$, $A\tens A$ and $A\tens\KK Q_1\tens A$.
\end{lem}
\begin{proof}
Each term of $\res{A}$ is a projective $A$-bimodule since $A$ is a projective $A$-module on each side, and is finitely generated by finiteness of $Q$. By standard results on presentations of algebras, see for example Butler--King \cite{butlerminimal}*{1.2}, $\mu_0$ is surjective, $\im{\mu_1}=\ker{\mu_0}$ and $\im{\mu_2}=\ker{\mu_1}$. Thus the only statement left to check is that $\mu_2\circ\mu_3=0$. 

Let $v\in Q_0$, and write
\[W_v=\sum_{\alpha\in\vout{v}}(\der{\alpha}{W})\alpha=\sum_{\beta\in\vin{v}}\beta(\der{\beta}{W}).\]
We can calculate $\sum_{\gamma\in Q_1}\Div{\gamma}(W_v)$ using each of the two expressions, to get
\begin{align*}
\sum_{\gamma\in Q_1}\Div{\gamma}(W_v)&=\sum_{\alpha\in\vout{v}}\sum_{\gamma\in Q_1}\Div{\gamma}(\der{\alpha}{W})\alpha+\sum_{\alpha\in\vout{v}}\der{\alpha}{W}\tens\alpha\tens 1,\\
\sum_{\gamma\in Q_1}\Div{\gamma}(W_v)&=\sum_{\beta\in\vin{v}}\sum_{\gamma\in Q_1}\beta\Div{\gamma}(\der{\beta}{W})+\sum_{\beta\in\vin{v}}1\tens\beta\tens\der{\beta}{W}.
\end{align*}
If $v\in Q_0^\mut$, then all arrows incident with $v$ are unfrozen, and so $\der{\alpha}{W}=0=\der{\beta}{W}$ in $A$ for any $\alpha\in\vout{v}$ and $\beta\in\vin{v}$. Thus in this case we have
\[\sum_{\alpha\in\vout{v}}\sum_{\gamma\in Q_1}\Div{\gamma}(\der{\alpha}{W})\alpha=\sum_{\gamma\in Q_1}\Div{\gamma}(W_v)=\sum_{\beta\in\vin{v}}\sum_{\gamma\in Q_1}\beta\Div{\gamma}(\der{\beta}{W}).\]
It follows that
\begin{align*}
\mu_2(\mu_3(1\tens\omega_v\tens 1))&=\mu_2\left(\sum_{\alpha\in\vout{v}}1\tens\rel{\alpha}\tens\alpha -\sum_{\beta\in\vin{v}}\beta\tens\rel{\beta}\tens 1\right)\\
&=\sum_{\alpha\in\vout{v}}\sum_{\gamma\in Q_1}\Div{\gamma}(\der{\alpha}{W})\alpha-\sum_{\beta\in\vin{v}}\sum_{\gamma\in Q_1}\beta\Div{\gamma}(\der{\beta}{W})\\
&=0.
\end{align*}
This completes the proof.
\end{proof}
If the map
\[\begin{tikzcd}[column sep=20pt]
A\tens\KK Q_3^\mut\tens A\arrow{r}{\mu_3}\arrow{d}&A\tens\KK Q_2^\mut\tens A\arrow{r}{\mu_2}\arrow{d}&A\tens\KK Q_1\tens A\arrow{r}{\mu_1}\arrow{d}&A\tens A\arrow{d}{\mu_0}\\
0\arrow{r}&0\arrow{r}&0\arrow{r}&A
\end{tikzcd}\]
from \Cref{frjacresiscomp} is a quasi-isomorphism, then $\res{A}$ is a projective bimodule resolution of $A$. This means that, for the presentation of $A$ as a frozen Jacobian algebra, with relations given by certain derivatives of the superpotential, the first syzygies are dual to the mutable vertices, and there are no higher syzygies. In particular, $\gldim{A}\leq 3$. By \Cref{frjacresiscomp}, this map is a quasi-isomorphism if and only if the cohomology of $\res{A}$ vanishes at $A\tens\KK Q_2^\mut\tens A$ and $A\tens\KK Q_3^\mut\tens A$ (cf.~\cite{broomheaddimer}*{Rem.~7.4}). We will usually abuse notation and denote the map $\res{A}\to A$ from \Cref{frjacresiscomp} by $\mu_0$.

If $F=\emptyset$, the map $\mu_0\colon\res{A}\to A$ being a quasi-isomorphism implies that $A$ is $3$-Calabi--Yau \citelist{\cite{ginzburgcalabiyau}*{Cor.~5.3.3}; \cite{broomheaddimer}*{Thm.~7.7}}. We now show that, in the general case, $\mu_0$ being a quasi-isomorphism implies that $A$ is bimodule internally $3$-Calabi--Yau with respect to $e$.

\begin{thm}
\label{frjaci3cy}
If $A$ is a frozen Jacobian algebra such that $\mu_0\colon\res{A}\to A$ is a quasi-isomorphism, then $A$ is bimodule internally $3$-Calabi--Yau with respect to the frozen idempotent $e=\sum_{v\in F_0}\idemp{v}$.
\end{thm}

\begin{proof}Since $\res{A}\in\per{\env{A}}$, the quasi-isomorphism $\mu_0\colon\res{A}\isoto A$ makes $\res{A}$ into a projective resolution of $A$, implying immediately that $\pdim_{\env{A}}A\leq3$ and $A\in\per{\env{A}}$. It remains to check condition (iii) from \Cref{bimodidcy}.

We begin by describing $\rhom{A}=\RHom_{\env{A}}(A,\env{A})\in\bdcat{\env{A}}$. Denoting $\Hom_{\env{A}}(-,-)$ by $(-,-)$, the complex $\rhom{A}$ is given by
\[\begin{tikzcd}[column sep=17pt]
(A\tens A,\env{A})\arrow{r}{-\contra{\mu_1}}&(A\tens\KK Q_1\tens A,\env{A})\arrow{r}{\contra{\mu_2}}&(A\tens\KK Q_2^\mut\tens{A},\env{A})\arrow{r}{-\contra{\mu_3}}&(A\tens\KK Q_3^\mut\tens A,\env{A})
\end{tikzcd}\]
with $\contra{\mu_i}\colon f\mapsto f\circ\mu_i$; see Keller \cite{kellercalabiyau}*{\S2.7} for the signs on the differentials.

There are $A$-bimodule isomorphisms $A\tens A\iso\bigdsum_{v\in Q_0}A\idemp{v}\tens_{\KK}\idemp{v}A$ and $\env{A}\iso A\tens_{\KK}A$. Introducing the shorthand notation
\[\mathbf{x}\tens\mathbf{y}=\sum_{i=1}^kx^i\tens y^i\]
for elements of $A\tens_\KK A$, a homomorphism $f_0\colon A\tens A\to\env{A}$ is uniquely determined by the values
\[f_0(1\tens\idemp{v}\tens 1)=\mathbf{x}_v\tens\mathbf{y}_v\]
for each $v\in Q_0$. Since $1\tens\idemp{v}\tens 1=\idemp{v}\tens\idemp{v}\tens\idemp{v}$, we must have
\[\mathbf{x}_v\tens\mathbf{y}_v=\idemp{v}\mathbf{x}_v\tens\mathbf{y}_v\idemp{v}\in\idemp{v}A\tens_\KK A\idemp{v},\]
but $\mathbf{x}_v$ and $\mathbf{y}_v$ may otherwise be chosen freely. If follows that we have an isomorphism
\begin{align*}
(A\tens A,\env{A})&\isoto A\tens\KK Q_3\tens A,&f_0&\mapsto\sum_{v\in Q_0}\mathbf{y}_v\tens\omega_v\tens\mathbf{x}_v
\intertext{of $A$-bimodules. Similar arguments yield explicit isomorphisms}
(A\tens\KK Q_1\tens A,\env{A})&\isoto A\tens\KK Q_2\tens A,&f_1&\mapsto\sum_{\alpha\in Q_1}\mathbf{y}_\alpha\tens\rel{\alpha}\tens \mathbf{x}_\alpha,\\
(A\tens\KK Q_2^\mut\tens A,\env{A})&\isoto A\tens\KK Q_1^\mut\tens A,&f_2&\mapsto\sum_{\alpha\in Q_1^\mut}\mathbf{y}'_\alpha\tens\alpha\tens \mathbf{x}'_\alpha,\\
(A\tens\KK Q_3^\mut\tens A,\env{A})&\isoto A\tens\KK Q_0^\mut\tens A,&f_3&\mapsto\sum_{v\in Q_0^\mut}\mathbf{y}'_v\tens \idemp{v}\tens \mathbf{x}'_v,
\end{align*}
where the functions $f_1$, $f_2$ and $f_3$ are uniquely determined by the values
\begin{align*}
f_1(1\tens\alpha\tens1)&=\mathbf{x}_\alpha\tens\mathbf{y}_\alpha\in\idemp{\head{\alpha}}A\tens_\KK A\idemp{\tail{\alpha}},\\
f_2(1\tens\rel{\alpha}\tens1)&=\mathbf{x}'_\alpha\tens\mathbf{y}'_\alpha\in\idemp{\tail{\alpha}}A\tens_\KK A\idemp{\head{\alpha}},\\
f_3(1\tens\omega_v\tens 1)&=\mathbf{x}'_v\tens\mathbf{y}'_v\in\idemp{v}A\tens_\KK A\idemp{v}.
\end{align*}
Since $\alpha\in F_1$ implies that $\head{\alpha},\tail{\alpha}\in F_0$, the map $\bar{\mu}_1\colon A\tens\KK Q_1\tens A\to A\tens\KK Q_0\tens A$ restricts to a map $A\tens\KK F_1\tens A\to A\tens\KK F_0\tens A$, and thus taking quotients yields a map $\dual{\mu_1}\colon A\tens\KK Q_1^\mut\tens A\to A\tens\KK Q_0^\mut\tens A$. Explicitly, $\dual{\mu_1}$ is given by
\[\dual{\mu_1}(1\otimes\alpha\otimes 1)=1\otimes(1-e)\alpha-\alpha(1-e)\otimes 1.\]
Define $\dual{\mu_2}$ to be the composition of $\bar{\mu}_2$ with the projection $A\tens\KK Q_1\tens A\to A\tens\KK Q_1^\mut\tens A$; explicitly
\[\dual{\mu_2}(1\tens\rel{\alpha}\tens 1)=\sum_{\beta\in Q_1^\mut}\Div{\beta}(\der{\alpha}{W}).\]
Finally, let $\dual{\mu_3}=\bar{\mu}_3$. Then one can check that the isomorphisms of $A$-bimodules defined above induce an isomorphism of $\rhom{A}$ with the complex
\[\begin{tikzcd}[column sep=20pt]
A\tens\KK Q_3\tens A\arrow{r}{\dual{\mu_3}}&A\tens\KK Q_2\tens A\arrow{r}{\dual{\mu_2}}&A\tens\KK Q_1^\mut\tens A\arrow{r}{\dual{\mu_1}}&A\tens\KK Q_0^\mut\tens A.
\end{tikzcd}\]
As an example to illustrate the necessary calculations, we show that we get an isomorphism of $\mu_2^*\colon(A\tens\KK Q_1\tens A,\env{A})\to(A\tens\KK Q_2^\mut\tens A,\env{A})$ with $\dual{\mu_2}\colon A\tens\KK Q_2\tens A\to A\tens\KK Q_1^\mut\tens A$. It suffices to check this on the generators $1\tens\rel{\alpha}\tens 1$ of $A\tens\KK Q_2\tens A$. First observe that under the isomorphism $(A\tens\KK Q_1\tens A,\env{A})\isoto A\tens\KK Q_2\tens A$, the preimage of $1\tens\rel{\alpha}\tens 1=\idemp{\tail{\alpha}}\tens\rel{\alpha}\tens\idemp{\head{\alpha}}$ is the $A$-bimodule homomorphism $f_\alpha$ determined by
\[f_\alpha(1\tens\beta\tens 1)=\begin{cases}\idemp{\head{\alpha}}\tens\idemp{\tail{\alpha}},&\beta=\alpha,\\0,&\otherwise.\end{cases}\]
We then calculate for each $\beta\in Q_1^\mut$ that
\begin{align*}
\mu_2^*(f_\alpha)(1\tens\rel{\beta}\tens 1)&=f_\alpha\mu_2(1\tens\rel{\beta}\tens 1)\\
&=f_\alpha\Big(\sum_{\gamma\in Q_1}\Div{\gamma}(\der{\beta}{W})\Big)\\
&=\mathbf{x}_\beta\tens\mathbf{y}_\beta,
\end{align*}
where
\[\Div{\alpha}(\der{\beta}{W})=\mathbf{x}_\beta\tens\alpha\tens\mathbf{y}_\beta.\]
We must then have
\[\Div{\beta}(\der{\alpha}{W})=\mathbf{y}_\beta\tens\beta\tens\mathbf{x}_\beta,\]
and so the isomorphism $(A\tens\KK Q_2^\mut\tens A,\env{A})\isoto A\tens\KK Q_1^\mut\tens A$ takes $\mu_2^*(f_\alpha)$ to
\[\sum_{\beta\in Q_1^\mut}\mathbf{y}_\beta\tens\beta\tens\mathbf{x}_\beta=\sum_{\beta\in Q_1^\mut}\Div{\beta}(\der{\alpha}{W})=\dual{\mu}_2(1\tens\rel{\alpha}\tens 1),\]
as required.

Now consider the commutative diagram
\begin{equation}
\label{rhomdiagram}
\begin{tikzcd}[column sep=20pt]
0\arrow{d}&0\arrow{d}&0\arrow{d}&0\arrow{d}\\
0\arrow{r}\arrow{d}&0\arrow{r}\arrow{d}&A\tens\KK F_1\tens A\arrow{r}\arrow{d}&A\tens\KK F_0\tens A\arrow{d}\\
A\tens\KK Q_3^\mut\tens A\arrow{r}{\mu_3}\arrow{d}{+}&A\tens\KK Q_2^\mut\tens A\arrow{r}{\mu_2}\arrow{d}{-}&A\tens\KK Q_1\tens A\arrow{r}{\mu_1}\arrow{d}{+}&A\tens\KK Q_0\tens A\arrow{d}{-}\\
A\tens\KK Q_3\tens A\arrow{r}{-\dual{\mu_3}}\arrow{d}&A\tens\KK Q_2\tens A\arrow{r}{-\dual{\mu_2}}\arrow{d}&A\tens\KK Q_1^\mut\tens A\arrow{r}{-\dual{\mu_1}}\arrow{d}&A\tens\KK Q_0^\mut\tens A\arrow{d}\\
A\tens\KK F_3\tens A\arrow{r}\arrow{d}&A\tens\KK F_2\tens A\arrow{r}\arrow{d}&0\arrow{r}\arrow{d}&0\arrow{d}\\
0&0&0&0\end{tikzcd}
\end{equation}
in which the columns are split exact, the second row is $\res{A}$, the third row is isomorphic to $\rhom{A}[3]$ by the preceding calculations, and the signs on the vertical arrows indicate whether the corresponding map is the inclusion or its negative.

The diagram (\ref{rhomdiagram}) provides us with a map of complexes $A\iso\res{A}\to\rhom{A}[3]$ in $\bdcat{\env{A}}$, and shows that the cone of this map has the form
\[\begin{tikzcd}[column sep=20pt]
C=A\tens\KK F_3\tens A\arrow{r}&A\tens(\KK F_2\dsum\KK F_1)\tens A\arrow{r}&A\tens\KK F_0\tens A.
\end{tikzcd}\]
Let $M\in\resdcat{\stab{A}}{A}$ have finite dimensional total cohomology. Since the cohomology of $M$ is concentrated in some interval, $M\in\bdcat{A}$. We pick a bounded representative $M^\bullet$ of the quasi-isomorphism class of $M$, allowing us to compute the complex $\RHom_A(C,M)$ as the total complex of the double complex with terms
\[\Hom_A(A\tens V_i\tens A,M^j),\]
where $V_1=\KK F_3$, $V_2=\KK F_2\dsum\KK F_1$, $V_3=\KK F_0$ and $V_i=0$ for all other $i$. Since each $S$-bimodule $V_i$ has the property that $eV_ie=V_i$, we have
\[A\tens V_i\tens A=Ae\tens V_i\tens eA,\]
so the terms of the relevant double complex are isomorphic to
\[\Hom_S(V_i\tens eA,\Hom_A(Ae,M^j))=\Hom_S(V_i\tens eA,eM^j).\]
Since $M\in\resdcat{\stab{A}}{A}$, the complex $eM^\bullet$ is acyclic. Moreover, since $S$ is semi-simple, $\Hom_S(V_i\tens eA,-)$ preserves acyclicity, so the vertical cohomology of the double complex vanishes. It follows that $\RHom_A(C,M)=0$. A similar argument shows that $\RHom_{A^{\op}}(C,N)=0$ for all $N\in\resdcat{\stab{A}^{\op}}{A^{\op}}$, so we conclude that $A$ is bimodule internally $3$-Calabi--Yau with respect to $e$.
\end{proof}

\begin{eg}
Since the frozen Jacobian algebra $A$ in \Cref{i3cyeg,frjacalgeg} is finite dimensional (of small dimension!)\ it is possible to check that $\res{A}\isoto A$ directly by choosing a basis of $A$, whence this algebra is bimodule internally $3$-Calabi--Yau with respect to this frozen idempotent. The algebra $A'$ in \Cref{i3cyeg} may also be presented as a frozen Jacobian algebra, and again the fact that $A'$ is finite dimensional allows us to check directly that $\res{A'}\isoto A'$, giving the bimodule internally $3$-Calabi--Yau property.

However, we also have many more examples, some of which are infinite dimensional. For example, let
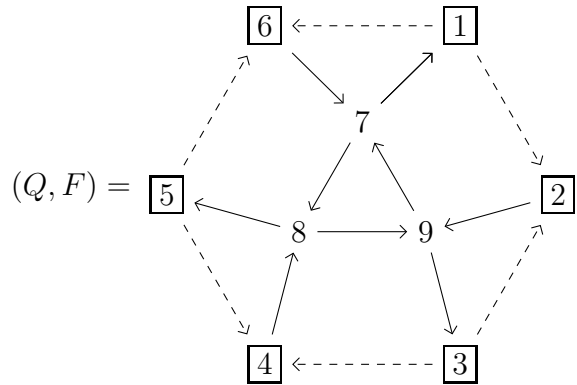
\[(Q,F)=\mathord{\begin{tikzpicture}[baseline=0,scale=1.3]
\node at (60:2) (1) {$\boxed{1}$};
\node at (0:2) (2) {$\boxed{2}$};
\node at (-60:2) (3) {$\boxed{3}$};
\node at (-120:2) (4) {$\boxed{4}$};
\node at (-180:2) (5) {$\boxed{5}$};
\node at (-240:2) (6) {$\boxed{6}$};
\node at (90:0.75) (7) {$7$};
\node at (-150:0.75) (8) {$8$};
\node at (-30:0.75) (9) {$9$};
\path[-angle 90,font=\scriptsize]
	(7) edge (8)
	(8) edge (9)
	(9) edge (7)
	(6) edge (7)
	(7) edge (1)
	(8) edge (5)
	(4) edge (8)
	(9) edge (3)
	(2) edge (9)
	(7) edge (1);
\path[-angle 90,font=\scriptsize,dashed]
	(1) edge (6)
	(5) edge (6)
	(5) edge (4)
	(3) edge (4)
	(3) edge (2)
	(1) edge (2);
\end{tikzpicture}}\]
where the frozen subquiver is indicated by boxed vertices and dashed arrows as before, and let $W$ be the potential given by the sum of the $3$-cycles minus the sum of the $4$-cycles. Writing $\hat{A}=\frjac{Q}{F}{W}$, we have $\res{\hat{A}}\isoto \hat{A}$. We omit the calculation here, but observe that it there is a grading of $\hat{A}$ in which all arrows have positive degree, meaning that by \cite{broomheaddimer}*{Prop.~7.5} it is sufficient to check that $\mu_0\tens_{\hat{A}}\simp{i}\colon\res{\hat{A}}\tens_{\hat{A}}\simp{i}\to\simp{i}$ is a quasi-isomorphism for each simple $\hat{A}$-module $\simp{i}$, which is more straightforward. One can also check that $\hat{A}$ is the endomorphism algebra of a cluster-tilting object in Jensen--King--Su's categorification of the cluster structure on the Grassmannian $\Grass{2}{6}$, described in \Cref{jkscats}.
\end{eg}

The existence of a quasi-isomorphism $\res{A}\isoto A$ allows us to deduce many homological properties of $A$ and of the boundary algebra $B$. For example, any $A$-module $M$ has a (usually non-minimal) projective resolution $\res{A}\tens_AM$. Using this, we see immediately that if $M$ is any $A$-module such that $M=eM$, such as a simple module at a frozen vertex, then $A\tens\KK Q_3^\mut\tens M=0$, and $\pdim_AM\leq 2$.

It follows from \Cref{frjaci3cy} that if $A$ is a frozen Jacobian algebra with the property that $\mu_0\colon\res{A}\to A$ is a quasi-isomorphism, then both $A$ and $A^{\op}$ are internally $3$-Calabi--Yau with respect to the frozen idempotent $e$. If additionally $A$ is Noetherian and $A/\Span{e}$ is finite dimensional, then we may apply \Cref{airanalogue}. We may also make some observations about the homological algebra of the boundary algebra $B$ in this situation.

\begin{prop}[cf.~\cite{amiotstable}*{Rem.~2.8}]
\label{rhomB}
Let $A$ be a Noetherian frozen Jacobian algebra with frozen idempotent $e$ such that $\mu_0\colon\res{A}\to A$ is a quasi-isomorphism and $A/\Span{e}$ is finite dimensional, and let $B=eAe$. Let $\rhom{B}=\RHom_{\env{B}}(B,\env{B})$. Then $\rhom{B}\iso e\rhom{A}e$ in $\bdcat{\env{B}}$.
\end{prop}
\begin{proof}
Write $P_i=A\tens\KK Q_i\tens A$ for $i=0,1$ and $P_i=A\tens\KK Q_i^\mut\tens A$ for $i=2,3$. By \Cref{frjaci3cy}, $A$ and $A^{\op}$ are internally $3$-Calabi--Yau with respect to $e$, so we have $\Ext^i_B(eA,B)=0=\Ext^i_{B^{\op}}(Ae,B)$ for all $i>0$ by \Cref{extcalcs}. Thus we may calculate
\begin{align*}
\RHom_{\env{B}}(eA\tens_\KK Ae,\env{B})&=\RHom_B(eA,B)\tens_\KK\RHom_{B^{\op}}(Ae,B)\\
&=\Hom_B(eA,B)\tens_\KK\Hom_{B^{\op}}(Ae,B)\\
&=\Hom_{\env{B}}(eA\tens_\KK Ae,\env{B}).
\end{align*}
It follows that the terms $eP_ie$ of the sequence $e\res{A}e\iso B$ satisfy $\Ext^i_{\env{B}}(eP_ie,\env{B})=0$ for $i>0$, and so
\[\RHom_{\env{B}}(B,\env{B})\iso\Hom_{\env{B}}(e\res{A}e,\env{B}).\]
By \Cref{airanalogue}(iii), the functor $eA\tens_A-\tens_AAe\colon\catproj{\env{A}}\to\fgmod{\env{B}}$ is fully faithful, and so
\[\Hom_{\env{B}}(eP_ie,\env{B})\iso\Hom_{\env{A}}(P_i,Ae\tens_\KK eA)=e\Hom_{\env{A}}(P_i,\env{A})e.\]
It follows that
\begin{align*}
\rhom{B}&=\RHom_{\env{B}}(B,\env{B})\iso\Hom_{\env{B}}(e\res{A}e,\env{B})\\
&=e\Hom_{\env{A}}(\res{A},\env{B})e\iso e\RHom_{\env{A}}(A,\env{A})e=e\rhom{A}e.\qedhere
\end{align*}
\end{proof}

\begin{prop}
\label{3cyesque}
With the notation and assumptions of \Cref{rhomB}, for any $X\in\bdcat{B}$ we have
\[\rhom{B}\ltens_BX\iso X[-3]\]
in the quotient category $\bdcat{B}/\per{B}\simeq\stabGP(B)$.
\end{prop}
\begin{proof}
The proof of \Cref{frjaci3cy} constructs a map $A\to\rhom{A}[3]$ with mapping cone
\[\begin{tikzcd}[column sep=20pt]
C=A\tens\KK F_3\tens A\arrow{r}&A\tens(\KK F_2\dsum\KK F_1)\tens A\arrow{r}&A\tens\KK F_0\tens A.
\end{tikzcd}\]
Since each $S$-bimodule $\KK F_i$ has the property that $e(\KK F_i)e=\KK F_i$, we can instead write $C$ as
\[\begin{tikzcd}[column sep=20pt]
Ae\tens\KK F_3\tens eA\arrow{r}&Ae\tens(\KK F_2\dsum\KK F_1)\tens eA\arrow{r}&Ae\tens\KK F_0\tens eA.
\end{tikzcd}\]
Now applying the functor $eA\tens_A-\tens_AAe$ to the triangle $A\to\rhom{A}[3]\to C\to A[1]$ in $\per{\env{A}}$ yields the triangle
\[\begin{tikzcd}[column sep=20pt]
B\arrow{r}&e\rhom{A}e[3]\arrow{r}&eCe\arrow{r}&B[1]
\end{tikzcd}\]
in $\bdcat{\env{B}}$. We have
\[\begin{tikzcd}[column sep=20pt]
eCe=B\tens\KK F_3\tens B\arrow{r}&B\tens(\KK F_2\dsum\KK F_1)\tens B\arrow{r}&B\tens\KK F_0\tens B\in\per{\env{B}},
\end{tikzcd}\]
and $e\rhom{A}e\iso\rhom{B}$ by \Cref{rhomB}. So applying $-\ltens_BX$ to the above triangle yields the triangle
\[\begin{tikzcd}[column sep=20pt]
X\arrow{r}&\rhom{B}\ltens_BX[3]\arrow{r}&eCe\ltens_B X\arrow{r}&X[1]
\end{tikzcd}\]
in $\bdcat{B}$. Since $eCe\in\per{\env{B}}$, we have $eCe\ltens_BX\in\per{B}$, and so $\rhom{B}\ltens_BX\iso X[-3]$ in the quotient $\bdcat{B}/\per{B}$, by the above triangle. The algebra $B$ is Iwanaga--Gorenstein by \Cref{airanalogue}(a), so $\bdcat{B}/\per{B}\simeq\stabGP(B)$ by \cite{buchweitzmaximal}*{Thm.~4.4.1}.
\end{proof}

\bibliographystyle{amsalpha}
\bibliography{../../mainbib}
\end{document}